\documentclass[12pt]{amsart}

\pdfoutput=1

\usepackage{microtype}

\usepackage{verbatim}
\usepackage{amssymb}
\usepackage{upref}
\usepackage[all]{xy}
\usepackage[usenames,dvipsnames]{color}
\usepackage{url}
\usepackage[normalem]{ulem}

\allowdisplaybreaks

\newtheorem{thm}{Theorem}[section]
\newtheorem*{thm*}{Theorem}
\newtheorem{lem}[thm]{Lemma}
\newtheorem{cor}[thm]{Corollary}
\newtheorem{prop}[thm]{Proposition}

\theoremstyle{definition}
\newtheorem{defn}[thm]{Definition}

\newtheorem{ex}[thm]{Example}
\newtheorem{notn}[thm]{Notation}

\newtheorem*{notn*}{Notation}

\newtheorem*{hyp*}{Hypothesis}

\newtheorem{rem}[thm]{Remark}
\newtheorem*{rem*}{Remark}

\numberwithin{equation}{section}

\newcommand{\secref}[1]{Section~\textup{\ref{#1}}}
\newcommand{\subsecref}[1]{Subsection~\textup{\ref{#1}}}
\newcommand{\thmref}[1]{Theorem~\textup{\ref{#1}}}
\newcommand{\corref}[1]{Corollary~\textup{\ref{#1}}}
\newcommand{\lemref}[1]{Lemma~\textup{\ref{#1}}}
\newcommand{\propref}[1]{Proposition~\textup{\ref{#1}}}

\newcommand{\defnref}[1]{Definition~\textup{\ref{#1}}}
\newcommand{\remref}[1]{Remark~\textup{\ref{#1}}}

\newcommand{\notnref}[1]{Notation~\textup{\ref{#1}}}

\newcommand{\midtext}[1]{\quad\text{#1}\quad}
\newcommand{\righttext}[1]{\quad\text{#1 }}
\renewcommand{\and}{\midtext{and}}

\renewcommand{\)}{\textup)}

\newcommand{\CC}{\mathcal C}
\newcommand{\DD}{\mathcal D}
\newcommand{\KK}{\mathcal K}

\newcommand{\LL}{\mathcal L}

\renewcommand{\epsilon}{\varepsilon}

\DeclareMathOperator{\aut}{Aut}

\DeclareMathOperator{\ad}{Ad}

\DeclareMathOperator{\mor}{Mor}
\DeclareMathOperator{\isos}{Iso}

\DeclareMathOperator*{\spn}{span}
\DeclareMathOperator*{\clspn}{\overline{\spn}}

\newcommand{\id}{\text{\textup{id}}}

\newcommand{\<}{\langle}
\renewcommand{\>}{\rangle}

\newcommand{\inv}{^{-1}}

\newcommand{\iso}{\overset{\simeq}{\longrightarrow}}
\newcommand{\variso}{\overset{\simeq}{\longrightarrow}}
\renewcommand{\bar}{\overline}
\newcommand{\what}{\widehat}
\newcommand{\wilde}{\widetilde}

\newcommand{\rt}{\textup{rt}}

\newcommand{\smtx}[1]{\left(\begin{smallmatrix} #1
\end{smallmatrix}\right)}
\newcommand{\mtx}[1]{\begin{pmatrix} #1 \end{pmatrix}}

\renewcommand{\:}{\colon}

\newcommand{\nd}{_\mathbf{nd}}
\newcommand{\en}{_\mathbf{en}}
\newcommand{\ou}{_\mathbf{ou}}
\newcommand{\du}{_\mathbf{sc}}

\newcommand{\cs}{\mathbf{C}^*}
\newcommand{\csn}{\cs\nd}
\newcommand{\cse}{\cs\en}

\newcommand{\ac}{\mathbf{Ac}}
\newcommand{\acn}{\ac_{\mathbf{nd}}}
\newcommand{\ace}{\ac_{\mathbf{en}}}
\newcommand{\aco}{\ac_{\mathbf{ou}}}

\newcommand{\wac}{\rt\text{-}{\ac}}
\newcommand{\wacn}{\wac\nd}
\newcommand{\wace}{\wac\en}
\newcommand{\waco}{\wac\ou}

\newcommand{\co}{\mathbf{Co}}
\newcommand{\con}{\co\nd}
\newcommand{\coe}{\co\en}
\newcommand{\coo}{\co\ou}

\newcommand{\wco}{\delta_G\text{-}{\co}}
\newcommand{\wcon}{\wco\nd}
\newcommand{\wcoe}{\wco\en}
\newcommand{\wcoo}{\wco\ou}
\newcommand{\wcod}{\wco\du}

\newcommand{\cp}{\textup{CP}}
\newcommand{\cpn}{\cp\nd}
\newcommand{\cpe}{\cp\en}
\newcommand{\cpo}{\cp\ou}

\newcommand{\wcp}{\wilde{\cp}}
\newcommand{\wcpn}{\wcp\nd}
\newcommand{\wcpe}{\wcp\en}
\newcommand{\wcpo}{\wcp\ou}
\newcommand{\wcpd}{\wcp\du}

\newcommand{\fix}{\textup{Fix}}
\newcommand{\fin}{\fix\nd}
\newcommand{\fie}{\fix\en}

\newcommand{\link}[3]{\mtx{#1&#2\\{*}&#3}}
\newcommand{\slink}[3]{\smtx{#1&#2\\{*}&#3}}

\renewcommand{\phi}{\varphi}
\renewcommand{\subset}{\subseteq}

\begin{document}
\title{Three versions of categorical crossed-product duality}

\author[Kaliszewski]{S.~Kaliszewski}
\address{School of Mathematical and Statistical Sciences
\\Arizona State University
\\Tempe, Arizona 85287}
\email{kaliszewski@asu.edu}
\author[Omland]{Tron Omland}
\address{School of Mathematical and Statistical Sciences
\\Arizona State University
\\Tempe, Arizona 85287}
\email{omland@asu.edu}
\author[Quigg]{John Quigg}
\address{School of Mathematical and Statistical Sciences
\\Arizona State University
\\Tempe, Arizona 85287}
\email{quigg@asu.edu}



\date{\today}

\begin{abstract}

In this partly expository paper we compare three different categories of $C^*$-algebras
in which crossed-product duality can be formulated, 
both for actions and for coactions of locally compact groups. 
In these categories, the isomorphisms correspond to
$C^*$-algebra isomorphisms, imprimitivity bimodules,
and outer conjugacies, respectively.

In each case, a variation of the fixed-point functor that arises from classical 
Landstad duality is used to obtain a quasi-inverse for a crossed-product functor.
To compare the various cases, 
we describe in a formal way
our view of the fixed-point functor as an ``inversion'' 
of the process of forming a crossed product.
In some cases, we obtain what we call ``good'' inversions,
while in others we do not.

For the outer-conjugacy categories, 
we generalize a theorem of Pedersen to obtain a fixed-point functor 
that is quasi-inverse to the reduced-crossed-product functor for actions, 
and we show that this gives a good inversion.
For coactions, we prove a partial version of Pedersen's theorem
that allows us to define a fixed-point functor, but 
the question of whether it is a
quasi-inverse for the crossed-product functor remains open.
\end{abstract}
\maketitle

\section{Introduction}\label{intro}

In crossed-product duality for $C^*$-algebras there are two problems that are of interest,
both stated for a fixed locally compact group $G$.
The first, and perhaps the original one, is:
Given a crossed product $A\rtimes_\alpha G$, 
how can $A$ and $\alpha$ be recovered?
Secondly: How can we identify a $C^*$-algebra $B$ 
as the (full or reduced) crossed product
of some other $C^*$-algebra $A$ by an action of $G$?

The
simplest case to consider is where $G$ is abelian.
Then there is a dual action of $\what G$ on $A\rtimes_{\alpha}G$ defined 
for $f\in C_c(G,A)$ by $\what\alpha_\chi(f)(t)=\chi(t)f(t)$.  
In this situation, a famous result of Takai \cite{takai} tells us that
\[
(A\rtimes_\alpha G)\rtimes_{\what\alpha}\what G \simeq A \otimes\KK(L^2(G)),
\]
so 
we can recover $A$ up to 
Morita equivalence
(and if $G$ is second countable, up to
stabilization).
The generalization of Takai's theorem to nonabelian groups involves the dual \emph{coaction} 
$\what\alpha$ of $G$.
This version is usually now called Imai-Takai duality \cite{it},
and in a similar fashion it gives an isomorphism
\[
(A\rtimes_{\alpha,r} G)\rtimes_{\what\alpha^n}G \simeq A \otimes\KK(L^2(G)),
\]
where $\what\alpha^n$ is our notation for the appropriate version of the dual coaction on the reduced crossed product.
However, Imai-Takai duality does not give any useful answer to the question of
when a $C^*$-algebra is a crossed product by an action of~$G$.
In fact, what the theorem says is that, {up to stabilization}, \emph{every} $C^*$-algebra is a crossed 
product.

For reduced crossed products by actions, Landstad answered both of the above questions up to isomorphism (\cite{lan:dual}).
First, given a crossed product $A\rtimes_{\alpha,r}G$,
we can recover $A$ as a generalized 
fixed-point algebra of the crossed product
that depends on both the dual coaction $\what\alpha^n$
of $G$ on $A\rtimes_{\alpha,r}G$ and the canonical embedding 
$i_G^r$ of $G$ into $M(A\rtimes_{\alpha,r}G)$.
Second, a given $C^*$-algebra $C$ is isomorphic
to a reduced crossed product by an action of $G$ if and only if
there exists a normal coaction of $G$ on $C$ and a unitary 
homomorphism of $G$ in $M(C)$ that interact with
one another like $\what\alpha^n$ and $i_G^r$ would.\footnote{Landstad used \emph{reduced} coactions,
but his results can be applied to full coactions 
on reduced crossed products using \cite{fullred}.}
The dual questions,
where actions are replaced by coactions,
were answered in \cite{qlandstad} (see \secref{class-coact}
for more details).

These results, now called \emph{classical Landstad duality},
have recently \cite{clda,cldx} been recast in  a categorical framework,
so let us first consider two categories of $C^*$-algebras and morphisms that are central in this context.

In the theory of $C^*$-algebras, and in particular in classification theory, there are  two types of equivalences that have an especially great impact; $C^*$-isomorphisms and Morita equivalences. Therefore, there are two categories of $C^*$-algebras that are natural to study; the \emph{nondegenerate category} $\csn$, whose morphisms are
nondegenerate $C^*$-homomorphisms,  and the \emph{enchilada category} $\cse$, whose morphisms are 
(isomorphism classes of) $C^*$-correspondences (see \secref{corres}).
The latter category does not have such a long history, and was treated extensively in \cite{enchilada}, although not with the name ``enchilada'' attached.
If we fix a locally compact group $G$, then the nondegenerate and the enchilada categories give rise to equivariant
categories $\acn$ and $\ace$, where the objects are pairs $(A,\alpha)$ comprising
a $C^*$-algebra $A$ and an action $\alpha$ of $G$ on $A$, and where the morphisms are the ones from $\csn$ or $\cse$, respectively, that are $G$-equivariant.

It was shown in \cite{clda} that the nondegenerate category of actions is equivalent to a certain \emph{comma category} of maximal coactions.
In that setup, a quasi-inverse functor from this comma category into the category of actions was constructed.
We call this the \emph{fixed-point functor}, since the image of an object in the comma category gives a generalized fixed-point algebra together with an action.

In this paper, we further develop this categorical perspective. In particular, a notion we call an \emph{inversion} of a functor $P\colon \CC\to \DD$ is introduced. Our motivation is that when $P$ is not an equivalence, we wish to keep track of what information it forgets. An inversion is therefore a category $\wilde \DD$ that contains the data of both $\DD$ and the extra structure that $P$ forgets, an equivalence $\wilde P\colon \CC\to\wilde \DD$, and a forgetful functor $F\colon\wilde \DD\to \DD$ with $F\circ\wilde P=P$.
Any choice of quasi-inverse $H\colon \wilde \DD\to \CC$ of $\wilde P$ is regarded as ``inverting the process'' $P$.

As we explain in \secref{invert actions},  \emph{categorical Landstad duality} fits into this setup. Indeed, the full-crossed-product functor $(A,\alpha)\to A\rtimes_\alpha G$ from $\acn$ to $\csn$ plays the role of $P\colon \CC\to \DD$, and the comma category plays the role of $\wilde \DD$. Moreover, this inversion is \emph{good}, meaning in particular that the forgetful functor $\wilde \DD\to \DD$ enjoys a certain lifting property.

In the same way, we find an inversion for the crossed-product functor between  enchilada categories. It turns out that the comma category analogous to the one used for the nondegenerate category has too few morphisms to be equivalent to the category of actions, so we need to consider a ``semi-comma category'' instead (borrowing a concept and terminology from \cite{hkrwnatural}). With this modification we get an inversion, and the quasi-inverse is a fixed-point functor. However, in this case the inversion is not good,
essentially because a $C^*$-algebra can be Morita equivalent to a crossed product without being isomorphic to one.

We remark that, as a consequence of the Imai-Takai duality mentioned above, a crossed-product functor that only keeps track of the dual coaction defines an equivalence between the enchilada category of actions and the enchilada category of coactions. This gives rise to an inversion as well. However, we want to compare the various categories and functors in a more direct way, so therefore the semi-comma category and the fixed-point functor are used.

We think of our third example as lying between the two cases discussed above. The underlying category in this example is still the nondegenerate one, but now the isomorphisms correspond to outer conjugacies. Inspired by a theorem of Pedersen, which we generalize from abelian to arbitrary groups, we define a certain ``fixed-point equivariant category'' of coactions, which is equivalent with the ``outer category'' of actions. This gives an inversion of the crossed-product functor from the outer category to $\csn$, which is also a good inversion.
The main innovation in this paper is the introduction and study of these outer categories.

The paper is organized as follows.
As an attempt to make it mostly self-contained,
we first provide a preliminary section recalling much of the background material.
Then, in \secref{ped thm} we prove the generalization of Pedersen's theorem for actions by nonabelian groups,
and also give a version for coactions.

Further, we introduce the category theoretical framework for inverting a process in \secref{abstract},
and define the concept of a (good) inversion.

In \secref{invert actions} we show that the three versions of crossed-product duality for actions fit into this categorical setup.
In particular, we show that the category equivalence arising from classical Landstad duality
gives category equivalences also in the enchilada and outer categories.
We present our results for full crossed products and the use of maximal coactions,
but only minor modifications are required to obtain similar results for the reduced-crossed-product functor.

In the last section,
for the nondegenerate and enchilada categories,
we produce abstract inversions of crossed-product duality for coactions similar to the ones for actions.
However, in this case, we work with normal coactions,
since this closely resembles the techniques applied for actions.

Finally, for the outer category,
we use a version Pedersen's theorem for coactions
that allows us to define a 
crossed-product functor 
in a manner parallel to the one for actions,
but our current version of Pedersen's theorem 
is not yet strong enough to give a category equivalence.

The second author would like to thank Johan Steen and Martin Wanvik from NTNU for helpful e-mail correspondence on various category-theoretic aspects.
The second author is funded by the Research Council of Norway (Project no.: 240913).

\section{Preliminaries}\label{prelim}

Throughout, $G$ will be a locally compact (Hausdorff) group.
By a homomorphism between $C^*$-algebras, we always mean a 
$*$-homomorphism.
We always use the minimal tensor product for $C^*$-algebras. 

\subsection{Actions and coactions}\label{act-coact}

An \emph{action} of $G$ on a $C^*$-algebra $A$
is a strongly continuous homomorphism $\alpha\colon G\to\aut(A)$.
Because we typically consider the group $G$ to be fixed and the actions to vary,
we will refer to the pair $(A,\alpha)$ as an action of $G$.
It is also common to call the triple $(A,G,\alpha)$ a \emph{$C^*$-dynamical system}.
One example of an action that deserves special mention is the \emph{right translation}
action $(C_0(G),\rt)$ defined by $\rt_sf(t)=f(ts)$.
Given a strictly continuous unitary homomorphism $V\colon G\to M(A)$,
which can equivalently be regarded as a nondegenerate homomorphism $V\colon C^*(G)\to M(A)$,
the associated \emph{inner action} $\ad u$ of $G$ on $A$ is defined by
\[
(\ad u)_s(a)=\ad u_s(a)=u_sau_s^*.
\]

To every action $(A,\alpha)$
we associate a full crossed product $A\rtimes_\alpha G$
and a reduced crossed product 
$A\rtimes_{\alpha,r}G$ in the usual way.
(A more detailed discussion of crossed products
can be found in \cite[Appendix~A]{enchilada}.)
We denote the  canonical universal
covariant homomorphism  of $(A,\alpha)$ in the multiplier algebra 
$M(A\rtimes_{\alpha} G)$ by $(i_A^\alpha,i_G^\alpha)$, 
and we write $\Lambda^\alpha\colon A\rtimes_\alpha G\to A\rtimes_{\alpha,r} G$ for the regular representation;
the canonical covariant homomorphism of $(A,\alpha)$ in 
$M(A\rtimes_{\alpha,r}G)$ is 
$(i_A^{\alpha,r}, i_G^{\alpha,r}) 
= (\Lambda^\alpha\circ i_A^\alpha,
\Lambda^\alpha\circ i_G^\alpha)$.  
However, when there is no potential ambiguity, 
we will abbreviate these as $(i_A,i_G)$, $\Lambda$, 
and $(i_A^r,i_G^r)$, respectively.
For every  covariant homomorphism $(\pi,U)$
of $(A,\alpha)$ in a $C^*$-algebra $C$,
there is an \emph{integrated form} $\pi\times U\colon
A\rtimes_\alpha G \to C$ such that
$(\pi\times U)\circ i_A = \pi$
and $(\pi\times U)\circ i_G = U$.
Moreover, if $\ker\Lambda \subseteq \ker(\pi\times U)$, 
then $\pi\times U$ descends to a homomorphism
$\pi\times_r U\colon A\rtimes_{\alpha,r}G\to C$,
also called the \emph{integrated form} of $(\pi,U)$,
such that $(\pi\times_{r}U)\circ\Lambda = \pi\times U$.

If $(A,\alpha)$ and $(B,\beta)$ are actions of~$G$,
a nondegenerate homomorphism $\phi\colon A\to M(B)$ is 
\emph{$\alpha-\beta$-equivariant} if
\begin{equation*}\label{G-equivariant}
\varphi\circ\alpha_s=\beta_s\circ\varphi\righttext{for all}s\in G.
\end{equation*}
Such a map induces nondegenerate homomorphisms 
$\phi\rtimes G\colon A\rtimes_\alpha G\to M(B\rtimes_\beta G)$
and $\phi\rtimes G\colon A\rtimes_{\alpha,r}G\to M(B\rtimes_{\beta,r}G)$.

Two actions $(A,\alpha)$ and $(B,\beta)$ are \emph{conjugate}
if there exists an $\alpha-\beta$-equivariant $C^*$-isomorphism 
$\varphi\colon A\to B$,
in which case
$\phi\rtimes G$ and $\phi\rtimes_r G$ 
are isomorphisms of the respective crossed products. 

A \emph{coaction} of $G$ on a $C^*$-algebra $A$
is an injective nondegenerate homomorphism $\delta\colon A\to M(A\otimes C^*(G))$ satisfying
the (additional) \emph{nondegeneracy condition}
\[
\clspn\{\delta(A)(1\otimes C^*(G))\}=A\otimes C^*(G)
\]
and the \emph{coaction identity}
\[
(\delta\otimes\id)\circ\delta=(\id\otimes\delta_G)\circ\delta.
\]
Here the coaction $\delta_G$ of $G$ on $C^*(G)$ is the  canonical map $C^*(G)\to M(C^*(G)\otimes C^*(G))$ given by the integrated form of $s\mapsto s\otimes s$.
In analogy with actions,
we also refer to the pair $(A,\delta)$ as
a coaction of $G$.
Given a nondegenerate homomorphism $\mu\colon C_0(G)\to M(A)$,
the associated \emph{inner coaction} $\ad\mu$ is given by
\[
\ad\mu(a)=\ad (\mu\otimes\id)(w_G)(a\otimes 1),
\]
where $w_G$ denotes the unitary element of
\[
M\bigl(C_0(G)\otimes C^*(G)\bigr)=C_b\bigl(G,M^\beta(C^*(G))\bigr)
\]
associated to the canonical unitary embedding of $G$ inside $M(C^*(G))$,
and where in turn $C_b(G,M^\beta(C^*(G)))$ denotes the continuous bounded functions from $G$ to $M(C^*(G))$ with the strict topology.

As with full crossed products by actions,
to each coaction $(A,\delta)$ we associate 
a crossed product $C^*$-algebra $A\rtimes_\delta G$,
and the covariant homomorphisms of $(A,\delta)$
correspond, via the integrated form, to homomorphisms
of $A\rtimes_\delta G$.  
The canonical universal covariant homomorphism of $(A,\delta)$ 
in $M(A\rtimes_\delta G)$
is  denoted by $(j_A^\delta,j_G^\delta)$, but as for actions, the notation is
usually simplified to avoid clutter.
When $j_A$ is injective,  $\delta$ is called a \emph{normal} coaction.

If $(A,\delta)$ and $(B,\epsilon)$ are coactions of $G$,
a nondegenerate homomorphism $\phi\colon A\to M(B)$ is 
$\delta-\epsilon$ \emph{equivariant}
if
\begin{equation}\label{G-equivariant-co}
(\varphi\otimes\id)\circ\delta=\epsilon\circ\varphi,
\end{equation}
and such a map
induces a nondegenerate homomorphism 
$\varphi\rtimes G\colon A\rtimes_\delta G\to M(B\rtimes_\epsilon G)$ 
between the corresponding crossed products.

Two  coactions $(A,\delta)$ and $(B,\epsilon)$ are \emph{conjugate}
if there exists a $\delta-\epsilon$ equivariant 
isomorphism $\varphi\colon A\to B$,
in which case
$\phi\rtimes G$ is an isomorphism of the
crossed products. 

For every action $(A,\alpha)$, there is a \emph{dual coaction}
$\what\alpha$ of $G$ on  $A\rtimes_{\alpha}G$,
defined on generators by
\begin{gather*}
\what\alpha(i_A(a))=i_A(a)\otimes 1 \quad\text{and}\quad
\what\alpha(i_G(s))=i_G(s)\otimes s.
\end{gather*}
There is also a normal dual coaction $\what\alpha^n$ 
on $A\rtimes_{\alpha,r}G$, defined similarly on generators.
Note that
$i_G\colon C^*(G)\to M(A\rtimes_\alpha G)$ is $\delta_G-\what\alpha$ equivariant;
it follows that if $(B,\beta)$ is an action and $\varphi\colon A\to M(B)$ is 
$\alpha - \beta$ equivariant,
then the induced homomorphism $\varphi\rtimes G\colon A\rtimes_{\alpha}G\to M(B\rtimes_{\beta}G)$
will be $\what\alpha - \what\beta$ equivariant,
and $\varphi\rtimes_r G\colon A\rtimes_{\alpha,r}G\to M(B\rtimes_{\beta,r}G)$
will be $\what\alpha^n - \what\beta^n$ equivariant.

Similarly, for every coaction $(A,\delta)$, 
there is a \emph{dual action} $\what\delta$ of $G$ on  $A\rtimes_\delta G$
defined by
\[
\what\delta_s=j_A\times(j_G\circ\rt_s).
\]
The canonical map $j_G\colon C_0(G)\to M(A\rtimes_\delta G)$ is $\rt - \what\delta$ equivariant, so 
if $(B,\epsilon)$ is a coaction and $\varphi\colon A\to M(B)$ is 
$\delta - \epsilon$ equivariant,
then the induced homomorphism $\varphi\rtimes G\colon A\rtimes_\delta G\to M(B\rtimes_\epsilon G)$
will be $\what\delta - \what\epsilon$ equivariant.

If $(A,\delta)$ is a coaction, then the pair
\[
\bigl((\id\otimes\lambda)\circ\delta\times (1\otimes M),1\otimes\rho\bigr),
\]
where $\lambda$ and $\rho$ are the left and right regular representations of $G$
and $M$ is the multiplication representation of $C_0(G)$ on $L^2(G)$,
is a covariant representation of the dual action $(A\rtimes_\delta G,\what\delta)$,
and the integrated form is a surjection
\[
\Phi\colon A\rtimes_\delta G\rtimes_{\what\delta} G\to A\otimes\KK(L^2(G)),
\]
called the \emph{canonical surjection},
where $\KK$ denotes the compact operators on $L^2(G)$.
The coaction $\delta$ is called \emph{maximal} if $\Phi$ is an isomorphism,
and by \cite[Theorem~2.2]{maximal} $\delta$ is normal if and only if $\Phi$ factors through an isomorphism
of the reduced crossed product
$A\rtimes_\delta G\rtimes_{\what\delta,r} G$
onto $A\otimes\KK(L^2(G))$.

\subsection{Normalization and maximalization}\label{normalization}
A \emph{normalization} of a coaction $(A,\delta)$ is a normal coaction $(A^n,\delta^n)$
together with a $\delta-\delta^n$ equivariant surjection $\eta\colon A\to A^n$ such that
\[
\eta\rtimes G\colon A\rtimes_\delta G\to A^n\rtimes_{\delta^n} G
\]
is an isomorphism.
Every coaction has a normalization,
and,
given another coaction $(B,\epsilon)$,
if $\phi\colon A\to M(B)$ is a nondegenerate $\delta-\epsilon$ equivariant homomorphism
then there is a unique nondegenerate $\delta^n-\epsilon^n$ equivariant homomorphism
$\phi^n$ making the following diagram commute:
\[
\xymatrix{
A \ar[r]^-\phi \ar[d]_{\eta_A}
&M(B) \ar[d]^{\eta_B}
\\
A^n \ar@{-->}[r]_-{\phi^n}
&M(B^n).
}
\]
Consequently, normalizations are unique up to isomorphism.

Similarly, a \emph{maximalization} of $(A,\delta)$ is a maximal coaction $(A^m,\delta^m)$
together with a $\delta^m-\delta$ equivariant surjection $\psi\colon A^m\to A$ such that
\[
\psi\rtimes G\colon A^m\rtimes_{\delta^m} G\to A\rtimes_\delta G
\]
is an isomorphism.
Every coaction has a maximalization,
and,
given another coaction $(B,\epsilon)$,
if $\phi\colon A\to M(B)$ is a nondegenerate $\delta-\epsilon$ equivariant homomorphism
then there is a unique nondegenerate $\delta^m-\epsilon^m$ equivariant homomorphism
$\phi^m$ making the following diagram commute:
\[
\xymatrix{
A^m \ar@{-->}[r]^-{\phi^m} \ar[d]_{\psi_A}
&M(B^m) \ar[d]^{\psi_B}
\\
A \ar[r]_-\phi
&M(B).
}
\]
Consequently, maximalizations are unique up to isomorphism.

If $(A,\delta)$ is a maximal coaction then
the normalization $\psi\colon A\to A^n$ is also a maximalization of the coaction $(A^n,\delta^n)$.
If $(B,\epsilon)$ is another maximal coaction,
then the map $\phi\mapsto \phi^n$ gives a bijection between
the sets of
$\delta-\epsilon$ equivariant nondegenerate homomorphisms $\phi\colon A\to M(B)$
and
$\delta^n-\epsilon^n$ equivariant nondegenerate homomorphisms $\phi^n\colon A^n\to M(B^n)$,
and moreover $\phi$ is an isomorphism if and only if $\phi^n$ is.

Given an action $(A,\alpha)$,
the dual coaction $\what\alpha$ on the full crossed product $A\rtimes_\alpha G$ is maximal,
the dual coaction $\what\alpha^n$ on the reduced crossed product $A\rtimes_{\alpha,r} G$ is normal,
and the regular representation
\[
\Lambda\colon (A\rtimes_\alpha G,\what\alpha)\to (A\rtimes_{\alpha,r} G,\what\alpha^n)
\]
is both a maximalization and a normalization.

\subsection{$C^*$-correspondences}\label{corres}

Let $A$ and $B$ be $C^*$-algebras.
An $A-B$ \emph{correspondence} is a (right) Hilbert $B$-module $X$ together with a homomorphism of $A$ 
into the $C^*$-algebra $\LL(X)$ of adjointable (hence bounded and $B$-linear) maps on $X$.
We say that the correspondence is \emph{nondegenerate} if $X$ is nondegenerate as a left $A$-module, i.e., $A\cdot X=X$.
For any $A-B$ correspondence $X$,
we use $M(X)$ to denote the set $\LL_B(B,X)$ of adjointable
maps from $B$ to $X$,
which is an $M(A)-M(B)$ correspondence in a natural way (see \cite[Definition~1.14]{enchilada}).

Given an $A-B$ correspondence $X$ and a $B-C$ correspondence~$Y$,
the balanced tensor product $X\otimes_B Y$ 
is an $A-C$ correspondence,
and the isomorphism class of $X\otimes_B Y$ depends only on the 
isomorphism classes of $X$ and $Y$ (\cite[Theorem~2.2]{enchilada}).

A \emph{Hilbert $A-B$ bimodule} is an $A-B$ correspondence that also has a left $A$-valued inner product ${}_A\<\cdot,\cdot\>$ that is compatible with the right $B$-valued inner product $\<\cdot,\cdot\>_B$ in the sense that
${}_A\<x,y\>\cdot z=x\cdot \<y,z\>_B$ for all $x,y,z\in X$.
An \emph{imprimitivity bimodule} is a Hilbert $A-B$ bimodule $X$ that is both left- and right-full,
meaning that $\clspn {}_A\<X,X\>=A$ and $\clspn\<X,X\>_B=B$.
Two $C^*$-algebras $A$ and $B$ are \emph{Morita equivalent} if there exists an 
$A-B$ imprimitivity  bimodule.

If $X$ is a nondegenerate $A-B$ correspondence,
$Y$ is a nondegenerate $C-D$ correspondence,
and $\phi\colon A\to M(C)$ and $\psi\colon B\to M(D)$ are homomorphisms,
a linear map $\zeta\colon X\to M(Y)$ is a
\emph{$\phi-\psi$ compatible correspondence homomorphism}
if
$\<\zeta(x),\zeta(y)\>_{M(D)}=\psi(\<x,y\>_B)$
and
$\phi(a)\cdot \zeta(x)=\zeta(a\cdot x)$
for all $x,y\in X$ and $a\in A$.
These properties imply that
$\zeta(x)\cdot \psi(b)=\zeta(x\cdot b)$
for all $x\in X$ and $b\in B$.
Sometimes we write
\[
(\phi,\zeta,\psi)\colon (A,X,B)\to (M(C),M(Y),M(D))
\]
for the correspondence homomorphism.
A correspondence homomorphism $(\phi,\zeta,\psi)$ is a \emph{correspondence isomorphism}
if $\phi\colon A\to C$ and $\psi\colon B\to D$ are $C^*$-isomorphisms and $\zeta\colon X\to Y$ is bijective.
In this case, if $X$ and $Y$ are Hilbert bimodules,
then $(\phi,\zeta,\psi)$ also preserves this extra structure in the sense that
\[
{}_C\<\zeta(x),\zeta(y)\>=\phi({}_A\<x,y\>)\righttext{for all}x,y\in X,
\]
and we call $(\phi,\zeta,\psi)$ a \emph{Hilbert bimodule isomorphism}.

Given actions $(A,\alpha)$ and $(B,\beta)$,
an $(A,\alpha)-(B,\beta)$ \emph{correspondence action} $(X,\gamma)$
is an $A-B$ correspondence $X$ equipped with an $\alpha-\beta$ 
compatible action $\gamma$ (\cite[Section~2.2]{enchilada}).
To every such correspondence action we associate 
a \emph{full crossed product} correspondence $X\rtimes_\gamma G$ that is an
$(A\rtimes_{\alpha} G)-(B\rtimes_{\beta} G)$ correspondence and
comes with a canonical universal
$i_A^\alpha-i_B^\beta$ compatible
correspondence homomorphism $i_X^\gamma$ of $X$
in $M(X\rtimes_\gamma G)$
such that
$X\rtimes_\gamma G=\clspn\{i_X^\gamma(X)\cdot i_G^\beta(C^*(G))\}$.
Similarly, there is a \emph{reduced crossed product} correspondence
$X\rtimes_{\gamma,r} G$
that is an $(A\rtimes_{\alpha,r} G)-(B\rtimes_{\beta,r} G)$ correspondence
and 
comes with a canonical 
$i_A^{\alpha,r}-i_B^{\beta,r}$ compatible
correspondence homomorphism $i_X^{\gamma,r}$.
Actions $(A,\alpha)$ and $(B,\beta)$ are \emph{Morita equivalent}
if there exists an 
$(A,\alpha)-(B,\beta)$ {correspondence action} $(X,\gamma)$
such that $X$ is an $A-B$ imprimitivity bimodule (\cite{combes}).

Given coactions $(A,\delta)$ and $(B,\epsilon)$,
an $(A,\delta)-(B,\epsilon)$ \emph{correspondence coaction}
is an $A-B$ correspondence $X$ equipped with a $\delta-\epsilon$ 
compatible coaction $\zeta$
(\cite[Section~2.3]{enchilada}).
For example, the crossed product correspondences
$X\rtimes_\gamma G$ and $X\rtimes_{\gamma,r}G$ described above
carry $\what\alpha-\what\beta$ and $\what\alpha^n-\what\beta^n$-compatible \emph{dual coactions} $\what\gamma$ and $\what\gamma^n$,
respectively.
To every correspondence coaction we associate a
\emph{crossed product} correspondence $X\rtimes_\zeta G$
that 
is an $(A\rtimes_\delta G)-(B\rtimes_\epsilon G)$ correspondence,
comes with a canonical 
$j_A^\delta-j_B^\epsilon$ compatible correspondence 
homomorphism $j_X^\zeta$ of
$X$ in $M(X\rtimes_\zeta G)$
such that
$X\rtimes_\zeta G=\clspn\{j_X^\zeta(X)\cdot j_G^\epsilon(C_0(G))\}$,
and carries
a $\what\delta-\what\epsilon$ compatible dual action $\what\zeta$. 
Two coactions $(A,\delta)$ and $(B,\epsilon)$ are  
\emph{Morita equivalent} if there 
exists an $(A,\delta)-(B,\epsilon)$ correspondence coaction 
$(X,\gamma)$ such that $X$ is an $A-B$ imprimitivity bimodule.

\subsection{Linking algebras}\label{linking}
Let $(A,\alpha)$ and $(B,\beta)$ be actions, let $(X,\gamma)$ be an $(A,\alpha)-(B,\beta)$ correspondence action,
let $K=\KK(X)$ be the algebra of generalized compact operators,
and let $L=L(X)=\smtx{K&X\\{*}&B}$
be the linking algebra
(see \cite[Section~1.5]{enchilada}).
Then by \cite[Proposition~2.27]{enchilada}
there is a unique action $\sigma$ of $G$ on $K$ such that $\gamma$ is $\sigma-\beta$ compatible,
and moreover the canonical nondegenerate homomorphism $\phi_A\colon A\to M(K)=\LL(X)$ is $\alpha-\sigma$ equivariant.
By \cite[Lemma~2.21]{enchilada} there is an action $\tau=\smtx{\sigma&\gamma\\{*}&\beta}$ of $G$ on $L$.
There is a natural identification (more properly, an isomorphism, but we blur the distinction)
\[
(L\rtimes_\tau G,\what\tau)=\left(\mtx{K\rtimes_\sigma G&X\rtimes_\gamma G\\{*}&B\rtimes_\beta G},\mtx{\what\sigma&\what\gamma\\{*}&\what\beta}\right).
\]
For the isomorphism of the crossed products, without the dual coactions,
see  \cite{combes,taco} --- these references require that the $B$-valued inner product be full, but the proof of the above isomorphism carries over.
\cite[Lemma~3.3 and Proposition~3.5 together with its proof]{enchilada} states the above isomorphism for reduced crossed products.

Dually, let $(A,\delta)$ and $(B,\epsilon)$ be coactions and let $(X,\zeta)$ be an $(A,\delta)-(B,\epsilon)$ correspondence coaction.
Then by \cite[Proposition~2.30]{enchilada}
there is a unique coaction $\mu$ of $G$ on $K$ such that $\zeta$ is $\mu-\epsilon$ compatible,
and moreover the canonical nondegenerate homomorphism $\phi_A\colon A\to M(K)=\LL(X)$ is $\delta-\mu$ equivariant.
By \cite[Lemma~2.22]{enchilada} there is a coaction $\nu=\smtx{\mu&\zeta\\{*}&\epsilon}$ of $G$ on $L$.
By \cite[Proposition~2.5]{maximal} $\mu$ and $\nu$ are maximal if $\epsilon$ is.
By \cite[Proposition~3.10]{enchilada} there is a natural identification (more properly, an isomorphism, but we blur the distinction)
\[
(L\rtimes_\nu G,\what\nu)=\left(\mtx{K\rtimes_\mu G&X\rtimes_\zeta G\\{*}&B\rtimes_\epsilon G},\mtx{\what\mu&\what\zeta\\{*}&\what\epsilon}\right).
\]
\cite[Proposition~3.10]{enchilada} only states this isomorphism for the crossed products; the statement regarding the dual actions
was apparently regarded in \cite{enchilada} as being self-evident.

\subsection{Exterior equivalence and outer conjugacy}\label{exterior}

Let $(B,\beta)$ be an action of $G$.
A \emph{$\beta$-cocycle} is a strictly continuous unitary map $u\colon G\to M(B)$ such that
\[
u_{st}=u_s\beta_s(u_t)\righttext{for all}s,t\in G.
\]
Given a $\beta$-cocycle $u$, the map $s\mapsto\ad u_s\circ\beta_s$ gives an action $\ad u\circ\beta$ on~$B$,
which is said to be \emph{exterior equivalent} to $\beta$.
An action $(A,\alpha)$ is \emph{outer conjugate} to $(B,\beta)$ if it is conjugate to $\ad u\circ\beta$ for 
some $\beta$-cocycle~$u$.

Now let $(B,\epsilon)$ be a coaction of $G$.
An \emph{$\epsilon$-cocycle} is a unitary element $U\in M(B\otimes C^*(G))$ such that
\begin{enumerate}
\item $(\id\otimes\delta_G)(U)=(U\otimes 1)(\epsilon\otimes\id)(U)$, and
\item $\ad U\circ\epsilon(B)(1\otimes C^*(G))\subset B\otimes C^*(G)$.
\end{enumerate}
Given an $\epsilon$-cocycle, $\ad U\circ\epsilon$ is a coaction on $B$ 
which is said to be \emph{exterior equivalent} to $\epsilon$,
and which is normal if $\epsilon$ is.
A coaction $(A,\delta)$ is \emph{outer conjugate} to $(B,\epsilon)$
if it is conjugate to $\ad U\circ\epsilon$ for some $\epsilon$-cocycle~$U$.

Of the three properties discussed in Subsections~\ref{act-coact}, \ref{corres}, and \ref{exterior},
conjugacy is stronger than outer conjugacy
(for both actions and coactions),
and outer conjugacy is in turn stronger than Morita equivalence.
Incidentally, ``outer'' Morita equivalence of actions or coactions, if it were defined
in analogy with \secref{exterior},
would just coincide with the respective type of equivariant
Morita equivalence.

\subsection{Classical Landstad duality for actions}\label{class-act}

As outlined in Section~\ref{intro}, 
Landstad duality is a method of recovering an action or coaction up to isomorphism from its crossed product,
as a ``generalized fixed-point algebra''.
Here we explain in more detail how this works for 
full crossed products by actions,
and also for
crossed products by normal coactions.

We will begin by recalling Landstad duality for reduced crossed products by actions.
\thmref{ldar} below is a reformulation of
\cite[Theorem~3.1]{ldfull},
modulo an addendum taken from
\cite[Theorem~3]{lan:dual}\footnote{Landstad used reduced, rather than full, coactions.}.

\begin{thm}[Landstad duality for reduced crossed products]\label{ldar}
Let $C$ be a $C^*$-algebra and $G$ a locally compact group.
Then there exist an action $(A,\alpha)$ and an isomorphism $\theta\colon A\rtimes_{\alpha,r} G\to C$
if and only if
there exist a normal coaction $\delta$ of $G$ on $C$
and a $\delta_G-\delta$ equivariant nondegenerate homomorphism $V\colon C^*(G)\to M(C)$.

Moreover, given $\delta$ and $V$ as above,
the action $(A,\alpha)$ and the isomorphism $\theta$ can be chosen such that
$\theta$ is $\what\alpha^n-\delta$ equivariant and
$\theta\circ i_G^r=V$;
with such a choice,
if $(B,\beta)$ is any action and
$\sigma\colon B\rtimes_{\beta,r} G\to C$ is a $\what\beta^n-\delta$ equivariant isomorphism such that $\sigma\circ i_G^r=V$,
then there exists an $\alpha-\beta$ equivariant isomorphism $\phi\colon A\to B$ such that
$\sigma\circ (\phi\rtimes_r G)=\theta$.

In fact,
we can take $A$ to be the $C^*$-subalgebra of $M(C)$
defined as all elements $a\in M(C)$ satisfying
\emph{Landstad's conditions}
\cite[(3.6)--(3.8)]{lan:dual}\textup{:}
\begin{align}
\label{fix}&\delta(a)=a\otimes 1;
\\
\label{in C}&\text{$aV(f),V(f)a\in C$ for all $f\in C_c(G)$;}
\\
\label{continuous}&\text{$s\mapsto \ad V_s(a)$ is norm continuous from $G$ to $C$,}
\end{align}
and we can let $\alpha$ be the restriction to $A$ of
\(the extension to $M(C)$ of\)
the inner action $\ad V$.
Then, letting $\iota\colon A\to M(C)$ be the inclusion map,
the pair $(\iota,V)$ is a covariant homomorphism of $(A,\alpha)$ in $M(C)$,
whose integrated form factors through an isomorphism $A\rtimes_{\alpha,r} G\simeq C$.
\end{thm}

In \thmref{lda} below we give a parallel version of \thmref{ldar} for full crossed products.
Some of the facts are
contained in 
\cite{ldfull} and
\cite{clda}.
The characterization in terms of Landstad's conditions seems to be new, however.

\begin{thm}[Landstad duality for full crossed products]\label{lda}
Let $C$ be a $C^*$-algebra and $G$ a locally compact group.
Then there exist an action $(A,\alpha)$ and an isomorphism $\theta\colon A\rtimes_\alpha G\to C$
if and only if
there exist a maximal coaction $\delta$ of $G$ on $C$
and a $\delta_G-\delta$ equivariant nondegenerate homomorphism $V\colon C^*(G)\to M(C)$.

Moreover, given $\delta$ and $V$ as above,
the action $(A,\alpha)$ and the isomorphism $\theta$ can be chosen such that
$\theta$ is $\what\alpha-\delta$ equivariant and
$\theta\circ i_G=V$;
with such a choice,
if $(B,\beta)$ is any action and
$\sigma\colon B\rtimes_\beta G\to C$ is a $\what\beta-\delta$ equivariant isomorphism such that $\sigma\circ i_G=V$,
then there exists an $\alpha-\beta$ equivariant isomorphism $\phi\colon A\to B$ such that
$\sigma\circ (\phi\rtimes G)=\theta$.

In fact,
we can take $A$ to be the $C^*$-subalgebra of $M(C)$
defined as all elements $a\in M(C)$ satisfying
Landstad's conditions \eqref{fix}--\eqref{continuous},
and we can let $\alpha$ be the restriction to $A$ of
\(the extension to $M(C)$ of\)
the inner action $\ad V$.
Then, letting $\iota\colon A\to M(C)$ be the inclusion map,
the pair $(\iota,V)$ is a covariant homomorphism of $(A,\alpha)$ in $M(C)$,
whose integrated form is an isomorphism $A\rtimes_\alpha G\simeq C$.
\end{thm}

\begin{proof}
The first two paragraphs are
\cite[Theorem~3.2]{ldfull}, 
modulo the slight improvement indicated in \cite[Remark~5.2]{clda}.
We must prove the third paragraph, involving Landstad's conditions,
and we combine techniques of the proofs of \cite[Lemma~3.1]{lan:dual} and \cite[Proposition~3.2]{qlandstad}:
It follows from the second paragraph of the theorem that
there is a $C^*$-subalgebra $A$ of $M(C)$ such that
$\ad V$ gives an action $\alpha$ of $G$ on $A$, and,
letting $\iota_A\colon A\to M(C)$ be the inclusion,
the pair $(\iota_A,V)$ is a covariant homomorphism of $(A,\alpha)$ in $M(C)$
whose integrated form is an isomorphism of $A\rtimes_\alpha G$ onto $C$.

Let
\[
B=\{a\in M(C):\text{Landstad's conditions \eqref{fix}--\eqref{continuous} hold}\}.
\]
Note that $A\subset B$.
Claim: $B$ is a $C^*$-subalgebra of $M(C)$.
Obviously the set of elements satisfying \eqref{fix} is a $C^*$-subalgebra.
For fixed $f\in C_c(G)$,
the set of elements $a\in M(C)$ such that $aV(f),V(f)a\in C$ is a closed subspace that is closed under adjoints,
and if it contains both $a$ and $b$ then $abV(f)\in C$ since $bV(f)\in C$,
and $V(f)ab\in C$ since $V(f)a\in C$.
Thus the claim is verified.

Note that $\ad V$ gives an action $\beta$ of $G$ on $B$, and,
letting $\iota_B\colon B\to M(C)$ be the inclusion,
the pair $(\iota_B,V)$ is a covariant homomorphism of $(B,\beta)$ in $M(C)$
whose integrated form $\iota_B\times V\colon B\rtimes_\beta G\to C$ is a $\what\beta-\delta$ equivariant surjective homomorphism.
Since $\iota_B$ is injective, by \cite[Corollary~4.4]{ldfull} 
$\iota_B\times V$ is an isomorphism.

Now let $\pi\colon A\to B$ be the inclusion.
Then $\pi$ is $\alpha-\beta$ equivariant,
and the induced homomorphism $\pi\rtimes G\colon A\rtimes_\alpha G\to B\rtimes_\beta G$ is an isomorphism
because
the following diagram commutes:
\[
\xymatrix@C+30pt{
A\rtimes_\alpha G \ar[r]^{\pi\rtimes G} \ar[dr]_{\iota_A\times V}^\simeq
&B\rtimes_\beta G \ar[d]^{\iota_B\times V}_\simeq
\\
&C.
}
\]
Taking crossed products by the dual coactions and applying crossed-product duality
(the statement of
\cite[Theorem~7]{rae:full}
is perhaps most suitable for the present purpose),
we get a commutative diagram
\[
\xymatrix@C+30pt{
A\rtimes_\alpha G\rtimes_{\what\alpha} G \ar[r]^-{\pi\rtimes G\rtimes G}_-\simeq \ar[d]^\simeq
&B\rtimes_\beta G\rtimes_{\what\beta} G \ar[d]_\simeq
\\
A\otimes \KK(L^2(G)) \ar[r]_-{\pi\otimes\id}
&B\otimes \KK(L^2(G)).
}
\]
Thus $\pi\otimes\id$, and hence $\pi$ itself, must be an isomorphism, and therefore $A=B$.
\end{proof}

The following definition applies to both \thmref{ldar} and \thmref{lda}.

\begin{defn}\label{fix coaction}
Let $\delta$ be a coaction of $G$ on a $C^*$-algebra $C$,
and let $V\colon C^*(G)\to M(C)$ be a $\delta_G-\delta$ equivariant nondegenerate homomorphism.
Then we call the triple $(C,\delta,V)$ an \emph{equivariant coaction}.
If $\delta$ is normal or maximal,
we denote the set of elements of $M(C)$ satisfying Landstad's conditions \eqref{fix}--\eqref{continuous} by $C^{\delta,V}$, or just $C^\delta$ if $V$ is understood, and we call this the \emph{generalized fixed-point algebra} of the equivariant coaction $(C,\delta,V)$.
Further, we write $\alpha^V$ for the action $\ad V$ on $C^{\delta,V}$.
\end{defn}

\begin{ex}
Starting with an action $(A,\alpha)$,
let $C=A\rtimes_\alpha G$, $\delta=\what\alpha$, and $V=i_G$.
Then
\[
i_A\colon A\to (A\rtimes_\alpha G)^{\what\alpha,i_G}\subset M(A\rtimes_\alpha G)
\]
is an $\alpha-\alpha^{i_G}$ equivariant isomorphism.
\end{ex}

\thmref{lda} immediately implies the following characterization of the image of $A$ in the multipliers of the full crossed product:

\begin{cor}\label{i_A(A)}
Let $(A,\alpha)$ be an action, and let $m\in M(A\rtimes_\alpha G)$. Then $m\in i_A(A)$ if and only if
\begin{enumerate}
\item $\what\alpha(m)=m\otimes 1$,
\item $mi_G(f),i_G(f)m\in A\rtimes_\alpha G$ for all $f\in C_c(G)$, and
\item $s\mapsto \ad i_G(s)(m)$ is norm continuous from $G$ to $M(A\rtimes_\alpha G)$.
\end{enumerate}
\end{cor}

We record a particular consequence of the above that we will need later:

\begin{cor}\label{LD}
Suppose $(C,\delta,V)$ is an equivariant maximal coaction
and $\phi\colon A\to C^{\delta,V}$ is an isomorphism.
Then there exist an action $\alpha$ of $G$ on $A$
and an $\what\alpha-\delta$ equivariant isomorphism
\[
\Theta\colon A\rtimes_\alpha G\variso C
\]
such that
\begin{align*}
\Theta\circ i_G&=V\\
\Theta\circ i_A&=\phi.
\end{align*}
\end{cor}
When we wish to appeal to \corref{LD} or any other aspect of the above discussion, we will just say ``by classical Landstad duality''.

\subsection{Classical Landstad duality for coactions}\label{class-coact}

The following result is a reformulation of
\cite[Theorem~3.3 and Proposition~3.2]{qlandstad}.

\begin{thm}
\label{ldcn}
Let $C$ be a $C^*$-algebra and $G$ a locally compact group.
Then there exist a normal coaction $(A,\delta)$ and an isomorphism $\theta\colon A\rtimes_\delta G\to C$
if and only if
there exist an action $\alpha$ of $G$ on $C$
and a $\rt-\alpha$ equivariant nondegenerate homomorphism $\mu\colon C_0(G)\to M(C)$.

Moreover, given $\alpha$ and $\mu$ as above,
the coaction $(A,\delta)$ and the isomorphism $\theta$ can be chosen such that
$\theta$ is $\what\delta-\alpha$ equivariant and
$\theta\circ j_G=\mu$;
with such a choice,
if $(B,\epsilon)$ is any normal coaction and
$\sigma\colon B\rtimes_\epsilon G\to C$ is a $\what\epsilon-\alpha$ equivariant isomorphism,
then there exists a $\delta-\epsilon$ equivariant isomorphism $\phi\colon A\to B$ such that
$\sigma\circ (\phi\rtimes G)=\theta$.

In fact,
we can take $A$ to be the unique $C^*$-subalgebra of $M(C)$ characterized by the 
following conditions,
modeled upon \cite[(3.1)-(3.3) in Proposition~3.2]{qlandstad}.
\begin{align}
\label{restrict} &\text{$\ad \mu$ restricts to a normal coaction on $A$;}
\\
\label{generate} &\clspn\{A\mu(C_0(G))\}=C;
\\
\label{fixed} &\text{$\alpha_s(a)=a$ for all $s\in G$ and $a\in A$,}
\end{align}
and we can let $\delta$ be the restriction to $A$ of
\(the extension to $M(C)$ of\)
the inner coaction $\ad \mu$.
Then, letting $\iota\colon A\to M(C)$ be the inclusion map,
the pair $(\iota,\mu)$ is a covariant homomorphism of $(A,\delta)$ in $M(C)$,
whose integrated form is an isomorphism $A\rtimes_\delta G\simeq C$.
\end{thm}

\begin{defn}\label{fix action}
Let $\alpha$ be an action of $G$ on a $C^*$-algebra $C$,
and let $\mu\colon C_0(G)\to M(C)$ be a $\rt-\alpha$ equivariant nondegenerate homomorphism.
Then we call the triple $(C,\alpha,\mu)$ an \emph{equivariant action}.
We denote the 
$C^*$-subalgebra
of $M(C)$ 
characterized by the
conditions \eqref{restrict}--\eqref{fixed} by $C^{\alpha,\mu}$, or just $C^\alpha$ if $\mu$ is understood, and we call this the \emph{generalized fixed-point algebra} of the equivariant action $(C,\alpha,\mu)$.
Further, we write $\delta^\mu$ for the coaction $\ad \mu$ on $C^{\alpha,\mu}$.
\end{defn}

\begin{ex}
Starting with a normal coaction $(A,\delta)$,
let $C=A\rtimes_\delta G$, $\alpha=\what\delta$, and $\mu=j_G$.
Then
\[
j_A\colon A\to (A\rtimes_\delta G)^{\what\delta,j_G}\subset M(A\rtimes_\delta G)
\]
is a $\delta-\delta^{j_G}$ equivariant isomorphism.
\end{ex}

\section{Pedersen's theorem}\label{ped thm}

Theorem~35 of \cite{pedersenexterior}, stated more precisely as \cite[Theorem~0.10]{raeros}, says that 
actions
$(A,\alpha)$ and $(A,\beta)$ of an \emph{abelian} group $G$ are exterior equivalent if and only if there is an $\what\alpha-\what\beta$ equivariant isomorphism $\Phi\colon A\rtimes_\alpha G\to A\rtimes_\beta G$ such that
\[
\Phi\circ i_A^\alpha=i_A^\beta.
\]

Pedersen's arguments carry over to the nonabelian case, except that we have to deal with the dual coaction of $G$ rather than the dual action of $\what G$.
Since we need it, and the coaction version for nonabelian groups does not seem to be readily available for reference in the literature, we include the statement and proof for completeness.

\begin{thm}
\label{pedersen}
Let $\alpha$ and $\beta$ be actions of a locally compact group $G$ on a $C^*$-algebra $A$.
Then $\alpha$ and $\beta$ are exterior equivalent
if and only if there is an $\what\alpha-\what\beta$ equivariant isomorphism $\Phi\colon A\rtimes_\alpha G\to A\rtimes_\beta G$
such that $\Phi\circ i_A^\alpha=i_A^\beta$.

Moreover, there is a bijection between the set of $\beta$-cocycles $u$ for which $\alpha=\ad u\circ\beta$
and the set of $\what\alpha-\what\beta$ equivariant isomorphisms $\Phi\colon A\rtimes_\alpha G\to A\rtimes_\beta G$
for which $\Phi\circ i_A^\alpha=i_A^\beta$,
given for $s\in G$ by
\begin{equation}\label{Phi}
\Phi\circ i_G^\alpha(s)=i_A^\beta(u_s)i_G^\beta(s).
\end{equation}
\end{thm}

\begin{proof}
First suppose that $u$ is a $\beta$-cocycle and $\alpha=\ad u\circ\beta$.
Define $V\colon G\to M(A\rtimes_\beta G)$ by
\[
V_s=i_A^\beta(u_s)i_G^\beta(s).
\]
Then $V$ is a strictly continuous unitary map, and is a homomorphism because $u$ is a $\beta$-cocycle.
Routine computations show that the pair $(i_A^\beta,V)$ is a covariant homomorphism of the action $(A,\alpha)$ in $M(A\rtimes_\beta G)$,
whose integrated form $\Phi$ takes $f\in C_c(G,A)$ to the element $\Phi(f)$ of $C_c(G,A)$ given by
\[
\Phi(f)(s)=f(s)u_s.
\]
Thus $\Phi$ maps $A\rtimes_\alpha G$ into $A\rtimes_\beta G$.
The $\alpha$-cocycle $u^*$ gives a homomorphism in the opposite direction that is easily verified to be an inverse of $\Phi$.

We verify that $\Phi$ is $\what\alpha-\what\beta$ equivariant
by checking the generators. For $a\in A$,
\begin{align*}
(\Phi\otimes\id)\circ \what\alpha\circ i_A^\alpha(a)
&=(\Phi\otimes\id)\bigl(i_A^\alpha(a)\otimes 1\bigr)
\\&=\Phi\circ i_A^\alpha(a)\otimes 1
\\&=i_A^\beta(a)\otimes 1
\\&=\what\beta\circ i_A^\beta(a)
\\&=\what\beta\circ \Phi\circ i_A^\alpha(a),
\end{align*}
and for $s\in G$,
\begin{align*}
(\Phi\otimes\id)\circ \what\alpha\circ i_G^\alpha(s)
&=(\Phi\otimes\id)\bigl(i_G^\alpha(s)\otimes s\bigr)
\\&=\Phi\circ i_G^\alpha(s)\otimes s
\\&=i_A^\beta(u_s)i_G^\beta(s)\otimes s
\\&=\bigl(i_A^\beta(u_s)\otimes 1\bigr)\bigl(i_G^\beta(s)\otimes s\bigr)
\\&=\what\beta\circ i_A^\beta(u_s)\what\beta\circ i_G^\beta(s)
\\&=\what\beta\bigl(i_A^\beta(u_s)i_G^\beta(s)\bigr)
\\&=\what\beta\circ \Phi\circ i_G^\alpha(s).
\end{align*}

Note that the above construction takes a
$\beta$-cocycle $u$ and produces
an $\what\alpha-\what\beta$ equivariant isomorphism
$\Phi\colon A\rtimes_\alpha G\to A\rtimes_\beta G$ such that
both
$\Phi\circ i_A^\alpha=i_A^\beta$
and \eqref{Phi} hold.
This construction is obviously injective from cocycles to equivariant isomorphisms.

Now suppose that $\Upsilon\colon A\rtimes_\alpha G\to A\rtimes_\beta G$ is an $\what\alpha-\what\beta$ equivariant isomorphism such that $\Upsilon\circ i_A^\alpha=i_A^\beta$.
Define $U\colon G\to M(A\rtimes_\beta G)$ by
\[
U_s=\Upsilon\bigl(i_G^\alpha(s)\bigr)i_G^\beta(s)^*.
\]
Then $U$ is a strictly continuous unitary map, and a quick calculation shows that for $s,t\in G$ we have
\begin{equation}\label{U cocycle}
U_{st}=U_s\ad i_G^\beta(s)(U_t).
\end{equation}
We claim that for all $s\in G$, the value $U_s$ is in the image $i_A^\beta(M(A))$.
First note that $i_A^\beta(M(A))=M(i_A^\beta(A))$.
So, we must
show that for every $a\in A$
the products $U_si_A^\beta(a)$ and $i_A^\beta(a)U_s$ are in $i_A^\beta(A)$.
We only give the argument for the first product;
the computations for the other product are very similar.

By \corref{i_A(A)},
it is enough to verify that the element $y=U_si_A^\beta(a)$ of $M(A\rtimes_\beta G)$ satisfies the following conditions:
\begin{enumerate}
\item $\what\beta(y)=y\otimes 1$;
\item $y\,i_G^\beta(f)$ and $i_G^\beta(f)\,y$ are in $A\rtimes_\beta G$ for all $f\in C_c(G)$;
\item $t\mapsto \ad i_G^\beta(t)(y)$ is norm continuous.
\end{enumerate}
For (i), we have
\begin{align*}
&\what\beta\bigl(\Upsilon\circ i_G^\alpha(s)i_G^\beta(s)^*i_A^\beta(a)\bigr)
\\&\quad=(\Upsilon\otimes\id)\circ \what\alpha\circ i_G^\alpha(s)\bigl(i_G^\beta(s)^*i_A^\beta(a)\otimes s\inv\bigr)
\\&\quad=(\Upsilon\otimes\id)\bigl(i_G^\alpha(s)\otimes s\bigr)\bigr)\bigl(i_G^\beta(s)^*i_A^\beta(a)\otimes s\inv\bigr)
\\&\quad=\Upsilon\circ i_G^\alpha(s)i_G^\beta(s)^*i_A^\beta(a)\otimes 1.
\end{align*}
For (ii), the first product is obvious, and the second product is similar once we note that
\begin{align*}
y
&=\Upsilon\circ i_G^\alpha(s)i_A^\beta\circ \beta_{s\inv}(a)i_G^\beta(s)^*
\\&=\Upsilon\circ i_G^\alpha(s)\Upsilon\circ i_A^\alpha\circ \beta_{s\inv}(a)i_G^\beta(s)^*
\\&=\Upsilon\bigl(i_A^\alpha\circ \alpha_s\circ \beta_{s\inv}(a)i_G^\alpha(s)\bigr)i_G^\beta(s)^*
\\&=i_A^\beta\circ \alpha_s\circ \beta_{s\inv}(a)U_s.
\end{align*}
Condition (iii) follows from combining
the identity $i_G^\beta(t)U_si_A^\beta(a)i_G^\beta(t)^*=\ad i_G^\beta(t)(U_s)i_A^\beta\circ \beta_t(a)$
with the facts that
$t\mapsto \ad i_G^\beta(t)(U_s)$ is strictly continuous and norm bounded, and
$t\mapsto i_A^\beta\circ \beta_t(a)$ is norm continuous.

We have proved the claim that $U_s\in i_A^\beta(M(A))$.
Since $i_A^\beta\colon A\to i_A^\beta(A)$ is an isomorphism,
we conclude that there is a unique strictly continuous unitary map $u\colon G\to M(A)$ such that
\[
U_s=i_A^\beta(u_s),
\]
and then since $i_A^\beta$ is $\beta-\ad i_G^\beta$ equivariant
it follows from \eqref{U cocycle} that $u$ is a $\beta$-cocycle.
By construction, the isomorphism $\Upsilon$ arises from this cocycle as in the first part of the proof,
and this proves the second part of the theorem.
\end{proof}

\thmref{pedersen} is the only version we will need.
However, we record the following version for reduced crossed products, since it might be useful elsewhere.

\begin{thm}
\label{pedersen reduced}
Let $\alpha$ and $\beta$ be actions of a locally compact group $G$ on a $C^*$-algebra $A$.
Then $\alpha$ and $\beta$ are exterior equivalent
if and only if there is an $\what\alpha^n-\what\beta^n$ equivariant isomorphism $\Phi\colon A\rtimes_{\alpha,r} G\to A\rtimes_{\beta,r} G$
such that $\Phi\circ i_A^{\alpha,r}=i_A^{\beta,r}$.

Moreover, there is a bijection between the set of $\beta$-cocycles $u$ for which $\alpha=\ad u\circ\beta$
and the set of $\what\alpha^n-\what\beta^n$ equivariant isomorphisms $\Phi\colon A\rtimes_{\alpha,r} G\to A\rtimes_{\beta,r} G$
for which $\Phi\circ i_A^{\alpha,r}=i_A^{\beta,r}$,
given for $s\in G$ by
\begin{align*}
\Phi\circ i_G^{\alpha,r}(s)&=i_A^{\beta,r}(u_s)i_G^{\beta,r}(s).
\end{align*}
\end{thm}

\begin{proof}
Recall from \subsecref{normalization} that the dual coaction $\what\alpha$ on a full crossed product $A\rtimes_\alpha G$ is maximal,
and
the regular representation
\[
\Lambda\colon (A\rtimes_\alpha G,\what\alpha)\to (A\rtimes_{\alpha,r} G,\what\alpha^n)
\]
is a both a normalization and a maximalization,
and consequently the map $\Phi\mapsto \Phi^n$ gives a bijection between the sets of 
$\what\alpha-\what\beta$ equivariant isomorphisms $\Phi\colon A\rtimes_\alpha G\to A\rtimes_\beta G$
and
$\what\alpha^n-\what\beta^n$ equivariant isomorphisms $\Phi\colon A\rtimes_{\alpha,r} G\to A\rtimes_{\beta,r} G$.

We need to know that
$\Phi\circ i_A^\alpha=i_A^\beta$
if and only if
$\Phi^n\circ i_A^{\alpha,r}=i_A^{\beta,r}$.
One direction is straightforward:
a computation
using the commutative diagram
\[
\xymatrix{
A\rtimes_\alpha G \ar[r]^-\Phi \ar[d]_{\Lambda^\alpha}
&A\rtimes_\beta G \ar[d]^{\Lambda^\beta}
\\
A\rtimes_{\alpha,r} G \ar[r]_-{\Phi^n}
&A\rtimes_{\beta,r} G
}
\]
shows that
$\Phi\circ i_A^\alpha=i_A^\beta$
implies
$\Phi^n\circ i_A^{\alpha,r}=i_A^{\beta,r}$.

On the other hand, the converse implication seems to be a little harder, requiring an indirect approach via cocycles again.
Assume that 
$\Phi^n\circ i_A^{\alpha,r}=i_A^{\beta,r}$.
Using the same technique as in the proof of \thmref{pedersen} for the case of full crossed products,
but working in the reduced crossed products,
we use the isomorphism $\Phi^n$ to
get a $\beta$-cocycle $u$ such that
\[
\Phi^n\circ i_G^{\alpha,r}(s)=i_A^{\beta,r}(u_s)i_G^{\beta,r}(s).
\]
Applying
\thmref{pedersen} to this cocycle $u$
gives an $\what\alpha-\what\beta$ equivariant isomorphism
$\Upsilon\colon A\rtimes_\alpha G\to A\rtimes_\beta G$
such that $\Upsilon\circ i_A^\alpha=i_A^\beta$
and $\Upsilon\circ i_G^\alpha(s)=i_A^\beta(u_s)i_G^\beta(s)$.
We check that
the diagram
\[
\xymatrix{
A\rtimes_\alpha G \ar[r]^-\Upsilon \ar[d]_{\Lambda^\alpha}
&A\rtimes_\beta G \ar[d]^{\Lambda^\beta}
\\
A\rtimes_{\alpha,r} G \ar[r]_-{\Phi^n}
&A\rtimes_{\beta,r} G
}
\]
 commutes
 by computing on the generators:
 \[
 \Phi^n\circ \Lambda^\alpha\circ i_A^\alpha
= \Phi^n\circ i_A^{\alpha,r}
=i_A^{\beta,r}
=\Lambda^\beta\circ i_A^\beta
=\Lambda^\beta\circ \Upsilon\circ i_A^\alpha,
 \]
 and for $s\in G$ we have
 \begin{align*}
 \Phi^n\circ \Lambda^\alpha\circ i_G^\alpha(s)
 &=\Phi^n\circ i_G^{\alpha,r}(s)
 =i_A^{\beta,r}(u_s)i_G^{\beta,r}(s)
 \\&=\Lambda^\beta\bigl(i_A^\beta(u_s)i_G^\beta(s)\bigr)
 =\Lambda^\beta\circ\Upsilon\circ i_G^\alpha(s).
 \end{align*}
Since the vertical maps $\Lambda^\alpha$ and $\Lambda^\beta$ are 
maximalizations,
the isomorphism $\Upsilon$ must coincide with $\Phi$.
Therefore $\Phi\circ i_A^\alpha=i_A^\beta$, as required.
\end{proof}

\begin{notn}
In 
\thmref{pedersen}, given a $\beta$-cocycle $u$, we will denote the associated isomorphism of $A\rtimes_\alpha G$ onto $A\rtimes_\beta G$ by $\Phi_u$.
\end{notn}

\begin{rem}
For actions $\alpha$ and $\beta$ on $A$, we could 
say that $\alpha$ is ``measurably exterior equivalent'' 
to $\beta$ if there exists a \emph{measurable} $\alpha$-cocycle~$\nu$ 
(in the sense that $t\mapsto \nu_t$ is measurable) such that $\beta=\ad\nu\circ\alpha$.
This is evidently a weaker notion than exterior equivalence.
However, 
we can adapt the technique of
proof of 
Proposition~\ref{pedersen} 
by defining $\Phi$ on $L^1(G,A)$ instead of $C_c(G,A)$, and then we can use 
\propref{pedersen}
to obtain a \emph{continuous} $\alpha$-cocycle $u$.
Thus, actions $\alpha$ and $\beta$ are ``measurably exterior equivalent'' if and only if they are exterior equivalent.
\end{rem}

The following elementary lemma is presumably folklore;
we include the proof because we could not find a reference for it in the literature.

\begin{lem}\label{cocycle multiply}
If $u$ is an $\alpha$-cocycle and $v$ is an $\ad u\circ\alpha$-cocycle,
then $vu$ is an $\alpha$-cocycle, 
and
\[
\Phi_{vu}=\Phi_u\circ\Phi_v.
\]
\end{lem}

\begin{proof}
First note that $vu\colon G\to M(A)$ is a strictly continuous unitary map, and for $s,t\in G$,
\begin{align*}
(vu)_{st}
&=v_{st}u_{st}
=v_s(\ad u\circ\alpha)_s(v_t)u_s\alpha_s(u_t)
\\&=v_s(\ad u_s\circ\alpha_s)(v_t)u_s\alpha_s(u_t)
\\&=v_su_s\alpha_s(v_t)u_s^*u_s\alpha_s(u_t)
\\&=v_su_s\alpha_s(v_t)\alpha_s(u_t)
\\&=v_su_s\alpha_s(v_tu_t)
=(vu)_s\alpha_s\bigl((vu)_t\bigr).
\end{align*}
Thus $vu$ is an $\alpha$-cocycle.

For the other part,
\begin{align*}
\Phi_u\circ \Phi_v
\circ i_A^{\ad v\circ\ad u\circ\alpha}
&=\Phi_u\circ i_A^{\ad u\circ\alpha}
=i_A^\alpha
=\Phi_{vu}\circ i_A^{\ad vu\circ\alpha},
\end{align*}
and for $s\in G$
\begin{align*}
\Phi_u\circ \Phi_v
\circ i_G^{\ad v\circ\ad u\circ\alpha}(s)
&=\Phi_u\bigl(i_A^{\ad u\circ\alpha}(v_s)i_G^{\ad u\circ\alpha}(s)\bigr)
\\&=i_A^\alpha(v_s)i_A^\alpha(u_s)i_G^\alpha(s)
\\&=i_A^\alpha\bigl((vu)_s\bigr)i_G^\alpha(s)
\\&=\Phi_{vu}
\circ i_G^{\ad vu\circ\alpha}(s)
\\&=\Phi_{vu}
\circ i_G^{\ad v\circ\ad u\circ\alpha}(s).
\qedhere
\end{align*}
\end{proof}

Proposition~2.8 of \cite{qrtwisted},
and its proof, imply the following partial analogue of Pedersen's theorem for coactions:

\begin{prop}[Pedersen's theorem for coactions]\label{copedersen}
Let $\delta$ be a normal coaction of $G$ on $A$.
Let $U$ be a $\delta$-cocycle, and let $\epsilon=\ad U\circ\delta$.
Then there is a unique $\what\epsilon-\what\delta$ equivariant isomorphism
\[
\Phi_U\colon A\rtimes_\epsilon G\variso A\rtimes_\delta G
\]
such that
\begin{align*}
\Phi_U\circ j_A^\epsilon&=j_A^\delta\\
(\Phi_U\otimes\id)\bigl((j_G^\epsilon\otimes\id)(w_G)\bigr)&=(j_A^\delta\otimes\id)(U)(j_G^\delta\otimes\id)(w_G).
\end{align*}
\end{prop}

However, for coactions the converse is still open.

The following elementary lemma,
dual to \lemref{cocycle multiply},
is presumably folklore,
and is included for completeness,
because the computations are peculiar to coactions.

\begin{lem}\label{cocycle multiply coaction}
If $U$ is a $\delta$-cocycle and $V$ is an $\ad u\circ\delta$ cocycle,
then $VU$ is a $\delta$-cocycle and
\[
\Phi_{VU}=\Phi_U\circ\Phi_V.
\]
\end{lem}

\begin{proof}
Clearly, $VU$ is a unitary in $M(A\otimes C^*(G))$,
and routine computations show that
$(\id\otimes\delta_G)(VU)=(VU\otimes 1)(\ad VU\circ\delta(VU))$
and
$\ad VU\circ \delta(A)(1\otimes C^*(G))\subset A\otimes C^*(G)$.
Thus $VU$ is a $\delta$-cocycle.

For the other part,
let $\epsilon=\ad U\circ\delta$ and $\zeta-\ad V\circ\epsilon$.
We have
\begin{align*}
\Phi_U\circ\Phi_V\circ j_A^\zeta
&=\Phi_U\circ j_A^\epsilon
=j_A^\delta
=\Phi_{VU}\circ j_A^\zeta,
\end{align*}
and
\begin{align*}
&(\Phi_U\circ\Phi_V\otimes\id)\bigl((j_G^\zeta\otimes\id(w_G)\bigr)
\\&\quad=(\Phi_U\otimes\id)\Bigl((\Phi_V\otimes\id)\bigl((j_G^\zeta\otimes\id)(w_G)\bigr)\Bigr)
\\&\quad=(\Phi_U\otimes\id)\bigl((j_A^\epsilon\otimes\id)(V)(j_G^\epsilon\otimes\id)(w_G)\bigr)
\\&\quad=(\Phi_U\circ j_A^\epsilon\otimes\id)(V)(\Phi_U\otimes\id)\bigl((j_G^\epsilon\otimes\id)(w_G)\bigr)
\\&\quad=(j_A^\delta\otimes\id)(V)(j_A^\delta\otimes\id)(U)(j_G^\delta\otimes\id)(w_G)
\\&\quad=(j_A^\delta\otimes\id)(VU)(j_G^\delta\otimes\id)(w_G).
\end{align*}
It follows that $\Phi_U\circ \Phi_V=\Phi_{VU}$.
\end{proof}

\section{Abstractly inverting a process}\label{abstract}

Very often in mathematics we are studying a process $P$ that takes inputs $x$ and produces outputs $P(x)$,
and several questions arise:
\begin{enumerate}
\item \label{outputs}
\textbf{Classify the outputs}: For which objects $y$ is there an input $x$ such that $P(x)=y$?
\item \label{inputs}
\textbf{Classify the inputs}: Given that $y$ is an output, find all inputs $x$ such that $P(x)=y$.
\item \label{invert}
\textbf{Invert the process}: Given that $P(x)=y$, what other data do we need to recover $x$?
\end{enumerate}
Typically we must interpret the above questions ``up to isomorphism''.
For example, we should write $P(x)\simeq y$ throughout,
in (ii) we should classify the $x$'s up to isomorphism,
and in (iii) we should only expect to recover $x$ up to isomorphism.

Our strategy is to put everything in a categorical setting,
so that the process $P$ is a functor,
and in (iii) we want to convert $P$ into a category equivalence.
We say ``convert'' here because most of the time the original version of the process will not be an equivalence.
When we inquire about ``other data'', which in mathematics are usually thought of as ``extra structure'', we want to factor $P$ as an equivalence followed by a specific type of forgetful functor, and the ``extra structure'' is what we are forgetting.

We want to give some meaning to ``inverting'' a functor $P$.
To begin, we introduce some terminology ---
some standard, some ad-hoc ---
that is convenient for our purposes.

\begin{defn}\label{invert def}
Suppose $P\colon \CC\to \DD$ is a functor.
\begin{itemize}
\item
An \emph{output} of $P$ is an object $y$ of $\DD$ for which there exists an object $x$ of $\CC$ such that $P(x)=y$.

\item
An \emph{essential output} of $P$ is an object $y$ of $\DD$ for which there exists an object $x$ of $\CC$ such that $P(x)\simeq y$.

\item
The \emph{image} of $P$ is the class of all outputs of $P$.

\item
The \emph{essential image} of $P$ is the class of all essential outputs of $P$.

\item
For an output $y$ of $P$,
the \emph{inverse image} of $y$ under $P$ is
the class $P\inv(y)$ of all objects $x$ of $\CC$ such that $P(x)=y$.

\item
For an essential output $y$ of $P$,
the \emph{essential inverse image} under $P$ of $y$ is
the class of all objects $x$ of $\CC$ such that $P(x)\simeq y$.

\item
If every object in $\DD$ is isomorphic to an output of $P$,
then $P$ is called \emph{essentially surjective}.

\item
If for all objects $x,y$ in $\CC$, the map $\mor(x,y)\to\mor(P(x),P(y))$
is surjective or injective,
then $P$ is \emph{full} or \emph{faithful}, respectively.

\item
If $P$ is essentially surjective, full, and faithful, then $P$ is a \emph{category equivalence} and has a \emph{quasi-inverse},
i.e., there is some functor $H\colon \DD\to \CC$ 
such that $H\circ P$ and $P\circ H$ are naturally isomorphic to the identity functors.

\item
$P$ is called \emph{conservative} if it reflects isomorphisms,
that is, if for every morphism $f$ in $\CC$ such that $P(f)$ is an isomorphism in $\DD$, then $f$ is an isomorphism in $\CC$.

\item
An \emph{inversion of $P$} is a commutative diagram
\[
\xymatrix{
\CC \ar[r]^{\wilde P} \ar[dr]_P
&\wilde \DD \ar[d]^F
\\
&\DD
}
\]
of functors such that
\begin{itemize}
\item[(i)] $\wilde P$ is an equivalence of categories;
\item[(ii)] $\wilde \DD$ is a category whose objects are pairs $(A,\sigma)$,
where $A$ is an object of $\DD$ and $\sigma$ denotes some extra structure;
\item[(iii)] $F$ is
defined by $F(A,\sigma)=A$ on objects, and is faithful.
\end{itemize}
\item
An inversion of $P$ (as above) is \emph{good}
if the image of $F$ is contained in the essential image of $P$,
and $F$ has the following
\emph{unique isomorphism lifting property}:
whenever $y\in \DD$ and $u\in F\inv(y)$,
for every isomorphism $\theta$ in $\DD$ with domain $y$
there is a unique isomorphism $\theta_u$ in $\wilde \DD$ with domain $u$
such that $F(\theta_u)=\theta$.
We write $\theta\cdot u$ for the codomain of $\theta_u$.
\end{itemize}
\end{defn}

We regard $F\colon\wilde \DD\to \DD$ as a special type of forgetful functor that ``forgets extra structure'',
and we regard any choice of quasi-inverse $H\colon \wilde \DD\to \CC$ of $\wilde P$ as ``inverting the process $P$''.

From a slightly different viewpoint, we can think about this as transforming $P$ into a forgetful functor $F$ by replacing its domain category $\CC$ by an equivalent category $\wilde \DD$, in a way that makes it clear what extra structure that is forgotten, i.e., what extra structure we need to invert the process.

We emphasize that the above is not an attempt to give a definition of inversion that is completely satisfactory to category theorists,
but we are rather describing a situation that is easily recognizable.
To shed further light on this, let us for a moment ignore part (ii) and (iii) of the definition of inversion,
and instead only require $F$ to be faithful,
so that it identifies $\mor(x,y)$ with a subset of $\mor(F(x),F(y))$.
Then one may think of an object $x$ in $\wilde \DD$ as having an underlying $\DD$-object $F(x)$,
plus some extra structure that $F$ forgets.
In this way,
one can think of morphisms $\mor(x,y)$ in $\wilde \DD$ as those morphisms in $\mor(F(x),F(y))$
that are ``compatible'' with this extra structure.
To simplify the description, we denote the extra structure by a symbol $\sigma$,
making the definition less rigorous.

Moreover, the requirements that $F$ is faithful and $\wilde P$ is an equivalence
mean that inversions only exist for functors $P$ that are faithful.
In fact, this is the only obstruction:
assuming $P\colon \CC\to \DD$ is faithful, define
a category $\wilde \DD$ whose objects comprise all pairs $(P(x),x)$ where $x$ is an object in $\CC$,
and in which a morphism from $(P(x),x)$ to $(P(y),y)$ is just a morphism $f\colon x\to y$ in $\CC$.
Define a functor $\wilde P\colon \CC\to \wilde \DD$ on objects by $\wilde P(x)=(P(x),x)$ and on morphisms by $\wilde P(f)=f$.
Then $\wilde P$ is actually an isomorphism of categories, 
and we have an inversion with
the forgetful functor $F\colon\wilde \DD\to \DD$ defined by $F(P(x),x)=P(x)$ and $F(f)=P(f)$.
This construction is rather artificial, and the point we wish to make in this paper is that inversions arise quite naturally, and can give useful information.

In a good inversion, the unique isomorphism lifting property implies that the forgetful functor $F$ will be
conservative.
It turns out that even for inversions that are not good the functor $F$ will frequently be conservative; for example, this will be the case in all the examples we consider in this paper.
Note that $F$ is conservative if and only if $P$ is conservative.

In general, the unique isomorphism lifting does not carry over from $F$ to $P$.
Indeed, suppose $F$ has this property.
For all $y \in \DD$, $u\in P\inv(y)$, and $\theta\in\isos_\DD(y,\cdot)$,
there only exists some $u' \simeq u$ and $\theta'\in\isos_\CC(u',\cdot)$ such that $P(\theta')=\theta$,
namely, $u'\in\wilde P\inv(\wilde P(u))$ and $\theta'\in\wilde P\inv(F(\theta_{\wilde P(u)}))$.
The isomorphism $u \to u'$ can be chosen in a canonical way for every choice of quasi-inverse $H$ for $\widetilde P$ by letting $u'=(H\circ\widetilde P)(u)$, and $\theta_{u'} = H(\theta_{\widetilde P(u)})$ is unique up to isomorphism.
In other words, for a good inversion to exist, $P$ must have a property that is very close to the unique lifting property in a category-theoretical sense (the requirement that the image of $F$ is contained in the essential image of $P$ does not impose any restrictions on what $P$ can be).

The unique isomorphism lifting property may of course be defined for \emph{any} functor,
and has presumably been studied,
but we could not find this precise property in the category theory literature.

The unique isomorphism lifting property is very close to the requirement that $F$ be a covering of the underlying groupoids,
except that we do not require that $F$ be surjective on objects.

Note that, even in good inversion, $F$ is not full on the underlying groupoids since the lift $\theta_u$ of $\theta$ does not necessarily belong to $\isos(u,v)$, but rather to $\isos(u,w)$ for some possibly different $w$ with $F(v)=F(w)$.

Suppose we have a good inversion of $P$.
We emphasize that we do \emph{not} assume that the objects of $\DD$ that are in the image of $F$ form a particularly large portion of the class of all objects of $\DD$; in particular, $F$ will typically not be essentially surjective, i.e., there typically will be objects of $\DD$ that are not isomorphic to anything in the image of $F$.
However, good inversion requires that the image of $F$ is \emph{isomorphism-closed} in the sense that any object of $\DD$ that is isomorphic to an output of $F$ is an output of $F$. It follows from this and the definition of good inversion that the essential image of $P$ coincides with the image of $F$.

\begin{prop}\label{good props}
Suppose we have a good inversion of $P$ as above.
Then for any essential output $y$ of $P$,
the essential inverse image of $y$ under $P$ is classified up to isomorphism
by the orbits of $F\inv(y)$ under the natural action of $\aut (y)$.
\end{prop}

\begin{proof}
For an object $x$ in $\CC$ we have
$P(x)\simeq y$
if and only if
there exists $u\in F\inv(y)$ such that $x\simeq H(u)$.
For two such $x_1,x_2$, with $x_i\simeq H(u_i)$, $u_i\in F\inv(y)$ ($i=1,2$),
we have
$x_1\simeq x_2$
if and only if
$u_1\simeq u_2$,
if and only if $u_2=\theta\cdot u_1$ for some automorphism $\theta$ of $y$.
\end{proof}

\section{Inverting the crossed-product process --- actions}\label{invert actions}

We will give three examples of (categorically) inverting the crossed-product process for actions.
In all three cases the objects
of the categories $\CC,\DD,\wilde \DD$
will remain the same, 
but 
in some sense the first example will have the fewest morphisms,
the second example the most,
and the third example somewhere in between.

Broadly speaking, we will start with a category $\CC$ of actions,
and the basic process will produce the full crossed product, 
which will be an object in a category $\DD$ of $C^*$-algebras,
and the  category $\wilde \DD$ will have objects comprising $C^*$-algebras with a coaction and a certain kind of equivariant map (see below).

In each example the objects of the category $\CC$ will 
be actions $(A,\alpha)$ of
$G$,
and the 
process $P$ will be 
a functor 
that takes an object
$(A,\alpha)$ 
to
the full crossed product $C^*$-algebra $A\rtimes_\alpha G$.
The objects of the 
category $\wilde \DD$ will be
equivariant maximal coactions
$(C,\delta,V)$,
i.e., 
$\delta$ is a maximal coaction of $G$ on a $C^*$-algebra $C$
and $V\colon C^*(G)\to M(C)$ is a nondegenerate $\delta_G-\delta$ equivariant homomorphism.
The  functor $\wilde P$ will take an object $(A,\alpha)$ to
$(A\rtimes_\alpha G,\what\alpha,i_G)$.

\subsection{Nondegenerate Landstad duality for actions}

This first example of inverting the process will be based upon the
\emph{nondegenerate category $\csn$ of $C^*$-algebras},
in which a morphism $\phi\colon C\to D$ is a nondegenerate homomorphism $\phi\colon C\to M(D)$.
A morphism $\phi$ is an isomorphism in the category if and only if it is a $C^*$-isomorphism in the usual sense.

The \emph{nondegenerate category $\acn$ of actions}
has
actions $(A,\alpha)$ of $G$ as objects,
and when we say
$\psi\colon (A,\alpha)\to (B,\beta)$ is a morphism in the category we mean
$\psi\colon A\to B$ is a morphism in $\csn$
that is $\alpha-\beta$ equivariant.
Isomorphisms in the category are equivariant $C^*$-isomorphisms.

The \emph{nondegenerate equivariant category $\wcon$ of coactions}
has
equivariant maximal coactions 
(see \defnref{fix coaction}) $(C,\delta,V)$ of $G$
as objects,
and when we say
$\psi\colon (C,\delta,V)\to (D,\epsilon,W)$ is a morphism in the category we mean
$\psi\colon C\to D$ is a morphism in $\csn$
that is $\delta-\epsilon$ equivariant
and satisfies
\[
W=\psi\circ V.
\]

The 
\emph{nondegenerate crossed-product functor $\cpn$} is given
on objects by
$(A,\alpha)\mapsto A\rtimes_\alpha G$,
and
on morphisms by
\[
\bigl(\phi\colon (A,\alpha)\to (B,\beta)\bigr)
\mapsto
\bigl(\phi\rtimes G\colon A\rtimes_\alpha G\to B\rtimes_\beta G\bigr),
\]
where we must keep in mind that $\phi\rtimes G$ is to be interpreted as a morphism in the nondegenerate category of $C^*$-algebras.

The \emph{nondegenerate equivariant crossed-product functor $\wcpn$} is given
on objects by
$(A,\alpha)\mapsto (A\rtimes_\alpha G,\what\alpha,i_G)$,
and
on morphisms by
\begin{multline*}
\bigl(\phi\colon (A,\alpha)\to (B,\beta)\bigr)
\mapsto
\\
\bigl(\phi\rtimes G\colon 
(A\rtimes_\alpha G,\what\alpha,i_G)\to (B\rtimes_\beta G,\what\beta,i_G\bigr),
\end{multline*}
where we recall that 
$\phi\rtimes G$ is $\what\alpha-\what\beta$ equivariant
and takes $i_G^\alpha$ to $i_G^\beta$.

The functor $\wcpn$ is an equivalence \cite[Theorem~5.1]{clda} and
$\cpn$ is the composition of $\wcpn$ followed by the forgetful functor $F\colon\wcon\to\csn$
defined on objects by $(C,\delta,V)\mapsto C$ and on morphisms by $f\mapsto f$.
Hence, $F$ is precisely the type of forgetful functor that fits into the framework of \secref{abstract},
and hence, this setup gives an inversion of the process $\cpn$.
We call this inversion \emph{nondegenerate Landstad duality for actions}.

Moreover, 
it follows from \cite[proofs of Theorems~4.1 and 5.1]{clda} that
a quasi-inverse of the nondegenerate equivariant crossed-product functor
is given by the \emph{nondegenerate fixed-point functor $\fin$}, defined
on objects by
$(C,\delta,V)\mapsto (C^{\delta,V},\alpha^V)$
(see \defnref{fix coaction} for notation),
and
on morphisms 
as follows:
if $\psi\colon (C,\delta,V)\to (D,\epsilon,W)$ is a morphism in $\wcon$,
then
\[
\fin(\psi)\colon (C^{\delta,V},\alpha^V)\to (D^{\epsilon,W},\alpha^W)
\]
is the unique morphism in $\acn$
such that
the diagram
\[
\xymatrix@C+30pt{
(C^{\delta,V}\rtimes_{\alpha^V} G,\what{\alpha^V},i_{C^{\delta,V}})
\ar[r]^-{\fin(\psi)\rtimes G} \ar[d]^\simeq
&(D^{\epsilon,W}\rtimes_{\alpha^W} G,\what{\alpha^W},i_{D^{\epsilon,W}})
\ar[d]_\simeq
\\
(C,\delta,V) \ar[r]_-\psi
&(D,\epsilon,W)
}
\]
commutes in $\wcon$,
where the vertical arrows are the canonical isomorphisms.

Since we have chosen the object map of $\fin$ to take an equivariant maximal coaction $(C,\delta,V)$ to the $C^*$-subalgebra $C^{\delta,V}$ of $M(C)$,
in our setting
the nondegenerate 
homomorphism
\[
\fin(\psi)\colon C^{\delta,V}\to M(D^{\epsilon,W})
\]
is the restriction of (the canonical extension to $M(C)$ of) $\psi$.
Thus,
the additional data required to recover the action from the full crossed product $A\rtimes_\alpha G$ consists of the dual coaction $\what\alpha$ and the canonical homomorphism $i_G$.

\begin{prop}\label{good nd}
The above nondegenerate Landstad duality is a good inversion.
\end{prop}

\begin{proof}
We must check 
that the image of $F$ is contained in the essential image of $\cpn$, and
the unique isomorphism lifting property.
The first follows immediately: if $C=F(C,\delta,V)$,
then by classical Landstad duality (\corref{LD}) there is an action $(A,\alpha)=\fin(C,\delta,V)$ such that $C\simeq \cpn(A,\alpha)$.
For the unique isomorphism lifting property,
given an object $(C,\delta,V)$ of $\wcon$
and an isomorphism $\theta\colon C\variso D$ in $\csn$,
we can use $\theta$ to carry the coaction $\delta$ and the homomorphism $V$
over to a coaction $\epsilon$ on $D$ and a $\delta_G-\epsilon$ equivariant nondegenerate homomorphism $W\colon C^*(G)\to M(D)$,
and then 
$\theta$ gives an isomorphism
\[
\wilde\theta\colon (C,\delta,V)\to (D,\epsilon,W)
\]
in $\wcon$ covering $\theta\colon C\to D$.
Since the forgetful functor is 
faithful
we see that $\wilde\theta$ is unique.
\end{proof}

\subsection{Enchilada Landstad duality for actions}

The \emph{enchilada category $\cse$ of $C^*$-algebras}
has
the same objects as $\csn$,
but now when we say
$[Y]\colon C\to D$ is a morphism in the category we mean
$[Y]$ is the isomorphism class of a nondegenerate $C-D$ correspondence $Y$.
Composition of morphisms is given by balanced tensor products,
and identity morphisms by the $C^*$-algebras themselves, viewed as correspondences in the standard way.
A morphism $[Y]$ is an isomorphism in the category if and only if $Y$ is an imprimitivity bimodule.

The \emph{enchilada category $\ace$ of actions}
has
the same objects as $\acn$,
but now when we say
$[X,\gamma]\colon (A,\alpha)\to (B,\beta)$ is a morphism in the category we mean
$[X]\colon A\to B$ in $\cse$
and $\gamma$ is an $\alpha-\beta$ compatible action of $G$ on $X$.
Isomorphisms in the category are equivariant Morita equivalences.

The \emph{enchilada equivariant category $\wcoe$ of coactions}
has
the same objects as $\wcon$,
but now when we say
$[Y,\zeta]\colon (C,\delta,V)\to (D,\epsilon,W)$ is a morphism in the category we mean
$[Y]\colon C\to D$ in $\cse$ and
$\zeta$ is a $\delta-\epsilon$ compatible coaction of $G$ on $Y$.
Note that this time the morphisms have nothing to do with the equivariant homomorphisms $V,W$.
In particular, isomorphisms in the category are just equivariant Morita equivalences.

While $\wcon$ defined in the previous subsection is a so-called comma category,
$\wcoe$ is sometimes loosely said to be a ``semi-comma category''.

The \emph{enchilada crossed-product functor $\cpe$} is
the same as $\cpn$
on objects,
but is given
on morphisms by
\[
\bigl([X,\gamma]\colon (A,\alpha)\to (B,\beta)\bigr)
\mapsto
\bigl([X\rtimes_\gamma G]\colon A\rtimes_\alpha G\to B\rtimes_\beta G\bigr).
\]

The \emph{enchilada equivariant crossed-product functor $\wcpe$} is 
the same as $\wcpn$
on objects,
but is given
on morphisms by
\begin{multline*}
\bigl([X,\gamma]\colon (A,\alpha)\to (B,\beta)\bigr)
\mapsto
\\
\bigl([X\rtimes_\gamma G,\what\gamma]\colon 
(A\rtimes_\alpha G,\what\alpha,i_G)\to (B\rtimes_\beta G,\what\beta,i_G\bigr).
\end{multline*}

Note that $\cpe$ is the composition of $\wcpe$ followed by the forgetful functor
$F\colon (C,\delta,V)\mapsto C$.

We will prove that, although $F$ and $\wcpe$ provide a way of inverting $\cpe$,
in this case we do not have a good inversion.
Propositions~\ref{fixed correspondence} and \ref{action up down} below, which we express in noncategorical terms, form the crux of the matter.

First we need ``generalized fixed-point correspondences":

\begin{prop}\label{fixed correspondence}
Let $(C,\delta,V)$ and $(D,\epsilon,W)$ be equivariant maximal coactions,
and let $(Y,\zeta)$ be a $(C,\delta)-(D,\epsilon)$ correspondence coaction.
Then there are a
$(C^\delta,\alpha^V)-(D^\epsilon,\alpha^W)$
correspondence action $(X,\gamma)$
and an isomorphism
\[
\Theta\colon (X\rtimes_\gamma G,\what\gamma)\simeq (Y,\zeta)
\]
of
$(C,\delta)-(D,\epsilon)$ correspondence coactions,
characterized by
\[
\Theta\bigl(i_X(x)\cdot i_G(d)\bigr)=x\cdot W(d)
\midtext{for}x\in X,d\in C^*(G).
\]
\end{prop}

\begin{proof}
Recall from \subsecref{linking}
there is an associated maximal coaction $\mu$ on the algebra $K:=\KK(Y)$ of generalized compact operators on $Y$, and a maximal coaction
$\nu=\smtx{\mu&\zeta\\{*}&\epsilon}$ on the linking algebra 
$L:=L(Y)=\smtx{K&Y\\{*}&D}$ (where we do not bother to explicitly indicate the lower-left corners, since they take care of themselves).
The composition
\[
U:=\varphi_C\circ V\colon C^*(G)\to M(K)
\]
(where $\varphi_C\colon C\to M(K)$ is the homomorphism associated to the left-module structure)
and the homomorphism $Z:=\smtx{U&0\\0&W}$ are equivariant from $\delta_G$ to $\mu$ and $\nu$, respectively.
The projections $p=\smtx{1&0\\0&0},q=\smtx{0&0\\0&1}\in M(L)$ are in the multiplier algebra of the generalized fixed-point algebra $L^\nu$, giving a matrix decomposition
\[
L^\nu=\mtx{K^\mu&X\\{*}&D^\epsilon},
\]
where we define $X=pL^\nu q$.
Thus $X$ is a $K^\mu-D^\epsilon$ Hilbert bimodule,
and hence a $C^\delta-D^\epsilon$ correspondence,
incorporating the nondegenerate homomorphism 
$\fin(\varphi_C)\colon C^\delta\to M(K^\mu)$.

Now, $L$ carries an inner action $\ad Z$.
The projections $p,q\in M(L)$ are $\ad Z$-invariant,
and the restrictions on the diagonal corners are
\[
\ad Z|_{pLp}=\ad U\midtext{and}\ad Z|_{qLq}=\ad W,
\]
so $\ad Z$ decomposes as a matrix of actions
\[
\ad Z=\mtx{\ad U&\ad (U,W)\\{*}&\ad W}
\]
(where $\ad (U,W)$ denotes the action of $G$ on $Y$ whose value at $s\in G$ is the operator $x\mapsto U_s\cdot x\cdot W_s^*$).
Thus the restriction $\alpha^Z=\ad Z|_{L^\nu}$ decomposes as a matrix
\[
\alpha^Z=\mtx{\alpha^U&\gamma\\{*}&\alpha^W}
\]
(where we define $\gamma_s=\alpha^Z_s|_X$ for $s\in G$).
Moreover, $\gamma$ is $\alpha^U-\alpha^W$ compatible,
and hence $\alpha^V-\alpha^W$ compatible,
incorporating 
$\fin(\varphi_C)$
again.

Classical Landstad duality (\corref{LD})
gives an isomorphism
\[
\theta_L\colon \bigl(L^\nu\rtimes_{\alpha^Z} G,\what{\alpha^Z}\bigr)\iso (L,\nu).
\]
On the other hand,
since the projection $p$ is $\alpha^Z$-invariant,
the projection $i_{L^\nu}(p)$ gives a matrix decomposition
\[
L^\nu\rtimes_{\alpha^Z} G=\mtx{K^\mu\rtimes_{\alpha^U} G&X\rtimes_\gamma G
\\{*}&D^\epsilon\rtimes_{\alpha^W} G}
\]
of full crossed products
and
\[
\what{\alpha^Z}=\mtx{\what{\alpha^U}&\what\gamma\\{*}&\what{\alpha^W}}
\]
of dual coactions.
Further, $\theta_L$ preserves the corner projections.
Thus $\theta_L$ restricts on the corners to a
$\what\gamma-\zeta$ equivariant
Hilbert-bimodule isomorphism
\[
(\theta_K,\Theta,\theta_D)\colon 
\bigl(K^\mu\rtimes_{\alpha^U} G,X\rtimes_\gamma G,D^\epsilon\rtimes_{\alpha^W} G\bigr)\variso (K,Y,D).
\]
We also have an $\what{\alpha^V}-\delta$ equivariant isomorphism 
$\theta_C\colon C^\delta\rtimes_{\alpha^V} G\variso C$,
and
\[\theta_K\circ \varphi_{C^\delta\rtimes_{\alpha^V} G}
=\varphi_C\circ \theta_C
\]
by nondegenerate Landstad duality.
Thus, incorporating the isomorphisms $\theta_C$ and $\theta_D$,
$\Theta$ is an isomorphism of $(C,\delta)-(D,\epsilon)$ correspondence coactions.

For the other part,
we have
\[
pi_{L^\nu}q=i_X
\midtext{and}qZ=qZ q=W,
\]
and it follows that for $x\in X,d\in C^*(G)$ we have
\[
\Theta\bigl(i_X(x)\cdot i_G(d)\bigr)=x\cdot W(d).
\qedhere
\]
\end{proof}

\begin{notn}\label{fix corr def}
We denote the $C^\delta-D^\epsilon$ correspondence $X$ constructed in the above proof
by $Y^{\zeta,V,W}$, or just $Y^\zeta$ if confusion is unlikely,
and we denote the action $\gamma$ by $\alpha^{V,W}$.
\end{notn}

\begin{prop}\label{action up down}
Let $(A,\alpha)$ and $(B,\beta)$ be actions,
and let $(X,\gamma)$ be an $(A,\alpha)-(B,\beta)$ correspondence action.
Then
$i_X\colon X\to M(X\rtimes_\gamma G)$ gives
an isomorphism
\[
(X,\gamma)
\iso 
\bigl((X\rtimes_\gamma G)^{\what\gamma},\alpha^{i_G^\alpha,i_G^\beta}\bigr)
\]
of $(A,\alpha)-(B,\beta)$ correspondence actions.
\end{prop}

\begin{proof}
We have associated actions
$(K=\KK(X),\sigma)$ and $(L=L(X),\tau)$,
with $\tau=\slink \sigma \gamma \beta$,
a dual coaction
\[
(L\rtimes_\tau G,\what\tau)
=\left(\link{K\rtimes_\sigma G}{X\rtimes_\gamma G}{B\rtimes_\beta G},
\link{\what\sigma}{\what\gamma}{\what\beta}\right)
\]
on the full crossed product,
and an isomorphism
\[
i_{L}\colon (L,\tau)\iso
\bigl((L\rtimes_\tau G)^{\what\tau},\alpha^{i_G^\tau}\bigr)
\]
by classical Landstad duality.
On the other hand, we also have a 
decomposition
\[
(L\rtimes_\tau G)^{\what\tau}
=\link
{(K\rtimes_\sigma G)^{\what\sigma}}
{(X\rtimes_\gamma G)^{\what\gamma}}
{(B\rtimes_\beta G)^{\what\beta}},
\]
and $i_{L}$ preserves the corner projections.
Thus $i_{L}$ restricts on the corners to
a
$\gamma-\alpha^{i_G^\sigma,i_G^\beta}$ equivariant
Hilbert-bimodule isomorphism
\[
(i_K,i_X,i_B)\colon 
(K,X,B)\iso
\bigl((K\rtimes_\sigma G)^{\what\sigma},
(X\rtimes_\gamma G)^{\what\gamma},
(B\rtimes_\beta G)^{\what\beta}\bigr).
\]
We also have an isomorphism $i_A\colon A\iso (A\rtimes_\alpha G)^{\what\alpha}$,
and
\[
i_K\circ \varphi_A=\varphi_{(A\rtimes_\alpha G)^{\what\alpha}}\circ i_A
\]
by nondegenerate Landstad duality.
Thus, incorporating the isomorphisms $i_A$ and $i_B$,
$i_X$ is an isomorphism of $(A,\alpha)-(B,\beta)$ correspondence actions.
\end{proof}

\begin{thm}\label{enchilada invert}
The enchilada equivariant crossed-product functor $\wcpe$ is an equivalence,
and there is a
quasi-inverse
$\fie\colon \wcoe\to \ace$
with
morphism map
\begin{multline*}
\bigl([Y,\zeta]\colon (C,\delta,V)\to (D,\epsilon,W)\bigr)
\\\mapsto
\bigl([Y^\zeta,\alpha^Y]\colon (C^\delta,\alpha^V)\to (D^\epsilon,\alpha^W)\bigr)
\end{multline*}
and the same object map as $\fin\colon \wcon\to \acn$.
\end{thm}

\begin{proof}
It is clear that $\wcpe\colon \ace\to \wcoe$ is essentially surjective, because it is essentially surjective for the nondegenerate categories, which have the same objects and in which isomorphism is stronger than in the enchilada categories.
To see that $\wcpe$ is an equivalence, we must show that,
for any two objects $(A,\alpha),(B,\beta)$ in $\ace$,
$\wcpe$ takes the set of morphisms
$\mor((A,\alpha),(B,\beta))$
bijectively onto
$\mor((A\rtimes_\alpha G,\what\alpha,i_G),(B\rtimes_\beta G,\what\beta,i_G))$.

For injectivity,
if $[X,\gamma]\colon (A,\alpha)\to (B,\beta)$ in $\ace$,
it suffices to note that 
\propref{action up down}
tells us that 
$(X,\gamma)$ can be recovered up to isomorphism from the crossed product.

We turn to the surjectivity.
If $[Y,\zeta]\colon (A\rtimes_\alpha G,\what\alpha,i_G)\to
(B\rtimes_\beta G,\what\beta,i_G)$
in $\wcoe$,
\propref{fixed correspondence} (and \notnref{fix corr def})
give an isomorphism
\[
\Theta\colon 
(Y^\zeta\rtimes_{\alpha^{i_G^\alpha,i_G^\beta}} G,\what{\alpha^{i_G^\alpha,i_G^\beta}})
\iso (Y,\zeta)
\]
of $(A\rtimes_\alpha G,\what\alpha)-(B\rtimes_\beta G,\what\beta)$
correspondence coactions.
Now, $(Y^\gamma,\alpha^{i_G^\alpha,i_G^\beta})$ is an
\[
\bigl((A\rtimes_\alpha G)^{\what\alpha},\alpha^{i_G^\alpha}\bigr)-
\bigl((B\rtimes_\beta G)^{\what\beta},\alpha^{i_G^\beta}\bigr)
\]
correspondence action.
Incorporating the isomorphisms
\begin{align*}
(A,\alpha)&\simeq \bigl((A\rtimes_\alpha G)^{\what\alpha},\alpha^{i_G^\alpha}\bigr)
\\(B,\beta)&\simeq \bigl((B\rtimes_\beta G)^{\what\beta},\alpha^{i_G^\beta}\bigr)
\end{align*}
from classical Landstad duality
(\thmref{lda}),
$(Y^\zeta,\alpha^{i_G^\alpha,i_G^\beta})$ becomes an $(A,\alpha)-(B,\beta)$ correspondence action,
whose 
full
crossed product is isomorphic to the given coaction $(Y,\zeta)$.
Finally, the assertions regarding the quasi-inverse $\fie$ follow immediately from the above constructions.
\end{proof}

Let $F\colon\wcoe\to\cse$ denote the forgetful functor defined on objects by $(C,\delta,V)\mapsto C$.
The factoring of $\cpn$ into a composition of $\wcpn$ followed by $F$ gives an inversion of $\cpe$,
which we call \emph{enchilada Landstad duality for actions}.

\begin{rem}
Enchilada Landstad duality for actions is \emph{not} a good inversion.
For if it were, the image of the forgetful functor
$(C,\delta,V)\mapsto C$ from $\wcoe$ to $\coe$
would coincide with the essential image of $\cpe$,
and it would follow that any $C^*$-algebra $C$ that is Morita equivalent to a full crossed product $A\rtimes_\alpha G$ would have extra structure $\delta,V$ such that $(C,\delta,V)$ is an object in $\wcoe$.
But then by classical Landstad duality $C$ would be $C^*$-isomorphic to a full crossed product.
We can easily see that this is false in general ---
for instance, if $G$ is finite of even order
then every finite-dimensional $C^*$-algebra isomorphic to a
crossed product by $G$ would have even dimension,
while every finite-dimensional $C^*$-algebra is Morita equivalent to one of odd dimension.

Nevertheless, it is still the case that 
the forgetful functor $F\colon \wcoe\to \cse$ is conservative, i.e.,
a morphism $[Y]\colon (C,\delta,V)\to (D,\epsilon,W)$ in $\wcoe$ is an isomorphism in the category if and only if
its image
$[Y]\colon C\to D$
under the forgetful functor
is an isomorphism in $\cse$, i.e., $Y$ is a $C-D$ imprimitivity bimodule
--- the problem is that isomorphisms do not always lift. 
Interestingly, the enchilada crossed-product functor $\cpe$ has a special property: it is essentially surjective, because
every $C^*$-algebra is Morita equivalent to a full crossed product,
by crossed-product duality.
In fact, we could use the dual crossed product
\[
(C,\delta,V)\mapsto (C\rtimes_\delta G,\what\delta)
\]
as an alternative quasi-inverse of $\wcpe$,
again by the properties of crossed-product duality.
This has the following consequence:
the map $(C,\delta,V)\mapsto (C,\delta)$ extends to an equivalence of enchilada categories that is actually surjective on objects.
\end{rem}

\subsection{Outer Landstad duality for actions}\label{outer duality actions}

We have seen that nondegenerate 
Landstad duality is a good inversion,
whereas enchilada Landstad duality is not.
In some sense the problem with the latter is that we have too many morphisms.
This led us to wonder whether there might be intermediate categories,
with more morphisms than the nondegenerate but fewer than the enchilada,
where good inversion is possible.
Here we present a nontrivial example of such an intermediate choice of morphisms.
However, this time the domain and target categories involve an asymmetrical choice of morphisms ---
in the domain category $\CC$ of actions we start with nondegenerate equivariant maps and throw in outer conjugacy,
while in the  category $\wilde \DD$
we require the coaction-equivariant maps to respect the generalized fixed-point algebras in some sense.

We base this third example of inverting the process upon the nondegenerate category $\csn$ of $C^*$-algebras, as we did for nondegenerate Landstad duality.
The \emph{outer category $\aco$ of actions}
has the same objects as $\acn$,
but now when we say $(\phi,u)\colon (A,\alpha)\to (B,\beta)$ is a morphism in the category
we mean $u$ is a $\beta$-cocycle and
$\phi\colon A\to B$ is a morphism in $\csn$
that is $\alpha-\ad u\circ \beta$ equivariant.

\begin{lem}\label{outer well defined}
The category $\aco$ introduced above is well-defined.
\end{lem}

\begin{proof}
The crucial thing is to check that we can compose morphisms:
let $(\phi,u)\colon (A,\alpha)\to (B,\beta)$ and $(\psi,v)\colon (B,\beta)\to (C,\gamma)$ be morphisms in $\aco$.
Claim:
\[
\bigl(\psi\circ\phi,(\psi\circ u)v\bigr)\colon (A,\alpha)\to (C,\gamma)
\]
is a morphism.
We need to show that:
\begin{enumerate}
\item $(\psi\circ u)v$ is a $\gamma$-cocycle.

\item $\psi\circ\phi$ is $\alpha-\ad\bigl((\psi\circ u)v\bigr)\circ \gamma$ equivariant.
\end{enumerate}

For (i), 
note that 
$\psi\circ u$ is an $\ad v\circ\gamma$ cocycle since $\psi$ is $\beta-\ad v\circ \gamma$ equivariant,
and hence it follows from Lemma~3.5 that $(\psi\circ u)v$ is a $\gamma$-cocycle.

For (ii), we reason as follows:
$\psi$ is $\beta-\ad v\circ\gamma$ equivariant
and $u$ is a $\beta$-cocycle,
so
$\psi\circ u$
is an $(\ad v\circ\gamma)$-cocycle
and $\psi$ is
\[
\ad u\circ\beta-\ad( \psi\circ u)\circ\ad v\circ\gamma
\]
equivariant.
Note that
\[
\ad (\psi\circ u)\circ\ad v\circ\gamma=\ad \bigl((\psi\circ u)v\bigr)\circ\gamma.
\]
Since $\phi$ is $\alpha-\ad u\circ\beta$ equivariant,
the composition $\psi\circ\phi$ is
\[
\alpha-\ad \bigl((\psi\circ u)v\bigr)\circ\gamma
\]
equivariant.

This proves the claim, and so composition of morphisms is well-defined.

It is obvious that there are identity morphisms,
and a routine computation shows that composition is associative.
\end{proof}

The isomorphisms in the category are just outer conjugacies of actions (and hence the name).

For \thmref{wcpo} below we will need
the \emph{fixed-point equivariant category $\wcoo$ of coactions},
which has the same objects as $\wcon$,
and in which
a morphism 
$\psi\colon (C,\delta,V)\to (D,\epsilon,W)$
is
a morphism $\psi\colon C\to D$ in $\csn$
that is $\delta-\epsilon$ equivariant
and satisfies
\begin{equation}\label{morphism}
D^{\epsilon,W}=D^{\epsilon,\psi\circ V}.
\end{equation}
However, there is a subtlety:
it is not obvious to us how to give a direct proof that composition
of the above morphisms
is well-defined.
We will in fact give an indirect argument for this
below (see \thmref{wcpo}).
To outline our strategy,
it will help to keep the following goal in mind:
we want to establish an equivalence
between the outer category of actions and the fixed-point equivariant category of coactions.
Due to the above difficulty, we will begin with a functor into an auxiliary category,
and the properties of this functor will allow us to eventually prove that it gives an equivalence with a subcategory.

Here is the auxiliary category:
the \emph{semi-comma equivariant category $\wcod$ of coactions}
has the same objects as $\wcon$,
namely equivariant maximal coactions,
and a morphism $\psi\:(C,\delta,V)\to (D,\epsilon,W)$ in the category is just
a morphism $\psi\:C\to D$ in $\csn$
that is $\delta-\epsilon$ equivariant.
The reason for the name ``semi-comma'' is that the morphisms have nothing to do with $V$ and $W$.
Note that
once we have cleared up the issue with compositions,
our desired category $\wcoo$ will be a subcategory of $\wcod$
obtained by keeping all the objects but placing a restriction on the morphisms.

Also note that the fixed-point condition on morphisms in $\wcoo$ does not say that $\psi$ takes the generalized fixed-point algebra $C^{\delta,V}$ to $D^{\epsilon,W}$,
but rather the two equivariant homomorphisms $W,\psi\circ V\colon C^*(G)\to M(D)$ give the same generalized fixed-point algebra.
However, an isomorphism in the category will preserve the generalized fixed-point algebras.

Let $(A,\alpha)$ be an action of $G$.
For an $\alpha$-cocycle $u$,
write
\[
\Phi_u\colon \bigl(A\rtimes_{\ad u\circ\alpha} G,\what{\ad u\circ\alpha}\bigr)
\variso (A\rtimes_\alpha G,\what\alpha\bigr)
\]
for the full-crossed-product isomorphism
given by Pedersen's theorem (\propref{pedersen}).

Suppose $(\phi,u)\colon (A,\alpha)\to (B,\beta)$ in $\aco$.
Then $u$ is a $\beta$-cocycle,
giving an exterior-equivalent action $\gamma=\ad u\circ\beta$,
and $\phi\colon (A,\alpha)\to (B,\gamma)$ in $\acn$.
Taking crossed products gives a morphism
\[
\phi\rtimes G\colon (A\rtimes_\alpha G,\what\alpha,i_G^\alpha)\to (B\rtimes_\gamma G,\what\gamma,i_G^\gamma)
\]
in the nondegenerate equivariant category $\wcon$ of coactions.
Forgetting some structure,
\[
\phi\rtimes G\colon A\rtimes_\alpha G\to B\rtimes_\gamma G
\]
is an $\what\alpha-\what\gamma$ equivariant morphism in $\csn$.

On the other hand, Pedersen's theorem gives a
$\what\gamma-\what\beta$ equivariant
isomorphism
\[
\Phi_u=i_B^\beta\times (i_B^\beta\circ u)i_G^\beta\colon B\rtimes_\gamma G\variso B\rtimes_\beta G
\]
in $\csn$.
Then composition gives an
$\what\alpha-\what\beta$ equivariant
morphism
\[
\Phi_u\circ (\phi\rtimes G)\colon A\rtimes_\alpha G\to B\rtimes_\beta G
\]
in $\csn$.

\begin{prop}\label{dumb functor}
With the above notation,
the assignments
\begin{align}
(A,\alpha)&\mapsto (A\rtimes_\alpha G,\what\alpha,i_G^\alpha)
\label{obj map}
\\
(\phi,u)&\mapsto \Phi_u\circ (\phi\rtimes G)
\label{mor map}
\end{align}
give a functor
$\wcpd\colon \aco\to \wcod$
that is faithful and essentially surjective.
\end{prop}

\begin{proof}
It follows immediately from the definitions that the object and morphism maps \eqref{obj map}--\eqref{mor map}
are well-defined and
that \eqref{mor map}
preserves identity morphisms.
To check that $\wcpd$ preserves compositions,
given morphisms
\[
\xymatrix{
(A,\alpha) \ar[r]^-{(\phi,u)} &(B,\beta) \ar[r]^-{(\psi,v)} &(C,\gamma)
}
\]
in $\aco$,
we have
\begin{align*}
&\wcpd(\psi,v)\circ \wcpd(\phi,u)\circ i_A
\\&\quad=\bigl(\Phi_v\circ (\psi\rtimes G)\bigr)
\circ
\bigl(\Phi_u\circ (\phi\rtimes G)\bigr)
\circ
i_A
\\&\quad=\Phi_v\circ (\psi\rtimes G)\circ \Phi_u\circ i_B^{\ad u\circ\beta}\circ \phi
\\&\quad=\Phi_v\circ (\psi\rtimes G)\circ i_B^\beta\circ \phi
\\&\quad=\Phi_v\circ i_C^{\ad v\circ\gamma}\circ \psi\circ \phi
\\&\quad=i_C^\gamma\circ \psi\circ \phi
\\&\quad=\Phi_{(\psi\circ u)v}\circ i_C^{\ad (\psi\circ u)v}\circ \psi\circ\phi
\\&\quad=\Phi_{(\psi\circ u)v}\circ \bigl((\psi\circ\phi)\rtimes G\bigr)\circ i_A
\\&\quad=\wcpd\bigl(\psi\circ\phi,(\psi\circ u)v\bigr)\circ i_A
\\&\quad=\wcpd\bigl((\psi,v)\circ (\phi,u)\bigr)\circ i_A,
\end{align*}
and for $s\in G$
\begin{align*}
&\wcpd(\psi,v)\circ \wcpd(\phi,u)\circ i_G^\alpha(s)
\\&\quad=\bigl(\Phi_v\circ (\psi\rtimes G)\bigr)\circ\bigl(\Phi_u\circ (\phi\rtimes G)\bigr)\circ i_G^\alpha(s)
\\&\quad=\Phi_v\circ (\psi\rtimes G)\circ \Phi_u\circ i_G^{\ad u\circ\beta}(s)
\\&\quad=\Phi_v\circ (\psi\rtimes G)\bigl(i_B^\beta(u_s)i_G^\beta(s)\bigr)
\\&\quad=\Phi_v\bigl(i_C^{\ad v\circ\gamma}\circ\psi(u_s)i_G^{\ad v\circ\gamma}(s)\bigr)
\\&\quad=i_C^\gamma\circ\psi(u_s)i_C^\gamma(v_s)i_G^\gamma(s)
\\&\quad=i_C^\gamma\bigl((\psi\circ u)_sv_s\bigr)i_G^\gamma(s)
\\&\quad=i_C^\gamma\Bigl(\bigl((\psi\circ v)u\bigr)_s\Bigr)i_G^\gamma(s)
\\&\quad=\Phi_{(\psi\circ u)v}\circ i_G^{\ad ((\psi\circ u)v\)\circ\gamma}(s)
\\&\quad=\Phi_{(\psi\circ u)v}\circ \bigl((\psi\circ\phi)\rtimes G\bigr)\circ i_G^\alpha(s)
\\&\quad=\wcpd\bigl(\psi\circ\phi,(\psi\circ u)v\bigr)\circ i_G^\alpha(s)
\\&\quad=\wcpd\bigl((\psi,v)\circ (\phi,u)\bigr)\circ i_G^\alpha(s).
\end{align*}
Thus $\wcpd\colon \aco\to \wcod$ is a functor.

It is clear that $\wcpd$ is essentially surjective, because
it is essentially surjective for the nondegenerate categories, which have the same objects,
and isomorphism in $\wcon$ is stronger than in $\wcod$.

To see that $\wcpd$ is faithful,
suppose that we have morphisms
\[
(\phi,u),(\rho,v)\colon (A,\alpha)\to (B,\beta)
\]
in $\aco$ such that
\[
\wcpd(\phi,u)=\wcpd(\rho,v)\colon 
(A\rtimes_\alpha G,\what\alpha,i_G^\alpha)\to
(B\rtimes_\beta G,\what\beta,i_G^\beta)
\]
in $\wcod$.
By construction, we have
\begin{align*}
\wcpd(\phi,u)\circ i_A^\alpha
&=\Phi_u\circ (\phi\rtimes G)\circ i_A^\alpha
\\&=\Phi_u\circ i_B^{\ad u\circ\beta}\circ \phi
\\&=i_B^\beta\circ \phi,
\end{align*}
and similarly
\[
\wcpd(\rho,v)\circ i_A^\alpha=i_B^\beta\circ \rho,
\]
so $\phi=\rho$ since $i_B^\beta$ is injective.

On the other hand, for $s\in G$ we have
\begin{align*}
\wcpd(\phi,u)\circ i_G^\alpha(s)
&=\Phi_u\circ (\phi\rtimes G)\circ i_G^\alpha(s)
\\&=\Phi_u\circ i_B^{\ad u\circ\beta}(s)
\\&=i_B^\beta(u_s)i_G^\beta(s),
\end{align*}
and similarly
\[
\wcpd(\rho,v)\circ i_G^\alpha(s)=i_B^\beta(v_s)i_G^\beta(s),
\]
so $u_s=v_s$ since $i_G^\beta(s)$ is unitary and $i_B^\beta$ is injective.
Thus $(\phi,u)=(\rho,v)$.
\end{proof}

With the above functor $\wcpd$ in hand, we can achieve our goal:

\begin{thm}\label{wcpo}
With the above notation,
the category $\wcoo$ is a well-defined subcategory of $\wcod$,
and the assignments \eqref{obj map}--\eqref{mor map} give
a category equivalence
$\wcpo\colon \aco\to \wcoo$.
\end{thm}

\begin{proof}
We will first show that 
for objects $(A,\alpha)$ and $(B,\beta)$ in $\aco$,
the functor $\wcpd$
gives a bijection 
\begin{align}\label{bij}
&\mor_{\aco}\bigl((A,\alpha),(B,\beta)\bigr)
\longleftrightarrow
\\
\notag
&\hspace{.5in}\bigl\{\psi\in \mor_{\wcod}\bigl(\wcpd(A,\alpha),\wcpd(B,\beta)\bigr):
\\
\notag
&\hspace{3in}\text{$\psi$ satisfies \eqref{morphism}}\bigr\}.
\end{align}

Given a morphism $(\varphi,u)\:(A,\alpha)\to (B,\beta)$ in $\aco$,
let $\gamma=\ad u\circ\beta$,
so that $\varphi\:(A,\alpha)\to (B,\gamma)$ is a morphism in $\acn$.
Since
$(\phi\rtimes G)\circ i_G^\alpha
=i_G^\gamma$
and $\Phi_u$ is a $\what\gamma-\what\beta$ equivariant isomorphism, we have
\begin{align*}
(B\rtimes_\beta G)^{\what\beta,i_G^\beta}
&=i_B^\beta(B)
\\&=\Phi_u(i_B^\gamma(B))
\\&=\Phi_u\bigl((B\rtimes_\gamma G)^{\what\gamma,i_G^\gamma}\bigr)
\\&=(B\rtimes_\gamma G)^{\what\beta,\Phi_u\circ i_G^\gamma}
\\&=(B\rtimes_\gamma G)^{\what\beta,\Phi_u\circ (\phi\rtimes G)\circ i_G^\alpha}
\\&=(B\rtimes_\gamma G)^{\what\beta,\wcpo(\phi,u)\circ i_G^\alpha},
\end{align*}
and hence $\wcpd(\varphi,u)$ satisfies \eqref{morphism}.

Now
suppose we are given a morphism
\[
\psi\colon (A\rtimes_\alpha G,\what\alpha,i_G^\alpha)\to (B\rtimes_\beta G,\what\beta,i_G^\beta)
\]
in $\wcod$
that satisfies \eqref{morphism}.
Put
\[
V=\psi\circ i_G^\alpha\colon G\to M(B\rtimes_\beta G).
\]
Since $\psi$ satisfies \eqref{morphism}, we have
\[
i_B^\beta(B)
=(B\rtimes_\beta G)^{\what\beta,i_G^\beta}
=(B\rtimes_\beta G)^{\what\beta,V}.
\]
Since $i_B^\beta\colon B\to i_B^\beta(B)$ is an isomorphism
and
$(B\rtimes_\beta G,\what\beta,V)$
is
an equivariant maximal coaction,
by classical Landstad duality (\corref{LD})
there are an action $\gamma$ of $G$ on $B$ and an isomorphism
\[
\Theta\colon (B\rtimes_\gamma G,\what\gamma,i_G^\gamma)\variso (B\rtimes_\beta G,\what\beta,V)
\]
in $\wcon$
such that
\begin{equation}\label{i}
\Theta\circ i_B^\gamma=i_B^\beta.
\end{equation}

Forgetting some structure, we have a $\what\gamma-\what\beta$ equivariant morphism
$\Theta\colon B\rtimes_\gamma G\to B\rtimes_\beta G$
satisfying
\eqref{i}, so by Pedersen's theorem 
there is a unique $\beta$-cocycle $u$
such that
$\gamma=\ad u\circ\beta$ and
$\Theta=\Phi_u$.

On the other hand, we can regard
\[
\psi\colon (A\rtimes_\alpha G,\what\alpha,i_G)\to (B\rtimes_\beta G,\what\beta,V)
\]
as a morphism
in $\wcon$,
and thus the composition
\[
\Theta\inv\circ \psi\colon (A\rtimes_\alpha G,\what\alpha,i_G)\to (B\rtimes_\gamma G,\what\gamma,i_G)
\]
is also a morphism
in $\wcon$.
By nondegenerate Landstad duality for actions discussed above \cite[Theorem~4.1]{clda},
there exists a unique morphism
\[
\phi\colon (A,\alpha)\to (B,\gamma)
\]
in $\acn$
such that $\Theta\inv\circ \psi=\phi\rtimes G$,
and then by construction of the functor $\wcpd$ we have
\begin{multline*}
\psi=\Theta\circ (\phi\rtimes G)=\Phi_u\circ (\phi\rtimes G)
=\wcpd(\phi,u)\colon 
\\
(A\rtimes_\alpha G,\what\alpha,i_G^\alpha)\to (B\rtimes_\beta G,\what\beta,i_G^\beta)
\end{multline*}
in $\wcod$.
Thus we have established the desired bijection \eqref{bij}.

Now let
\begin{align*}
\psi&\:(D,\delta,V)\to (E,\epsilon,W)
\\
\rho&\:(E,\epsilon,W)\to (F,\zeta,U)
\end{align*}
be morphisms in $\wcod$ satisfying \eqref{morphism}.
We must show that the composition $\rho\circ\psi$ also satisfies \eqref{morphism}.
By nondegenerate Landstad duality we have
actions
$(A,\alpha)$, $(B,\beta)$, and $(C,\gamma)$ such that
\begin{align*}
(A\rtimes_\alpha G,\what\alpha,i_G^\alpha)&\simeq (D,\delta,V)\\
(B\rtimes_\beta G,\what\beta,i_G^\beta)&\simeq (E,\epsilon,W)\\
(C\rtimes_\gamma G,\what\gamma,i_G^\gamma)&\simeq (F,\zeta,U),
\end{align*}
where the isomorphisms take place in the nondegenerate equivariant category $\wcon$,
i.e., there are
$C^*$-isomorphisms
\begin{align*}
\sigma&\:A\rtimes_\alpha G\variso D\\
\tau&\:B\rtimes_\beta G\variso E\\
\omega&\:C\rtimes_\gamma G\variso F
\end{align*}
that are equivariant for the dual coactions and the given coactions $\delta$, $\epsilon$, and $\zeta$, respectively,
and that also satisfy
\[
\sigma\circ i_G^\alpha=V,\quad
\tau\circ i_G^\beta=W,\midtext{and}
\omega\circ i_G^\gamma=U.
\]
Under these isomorphisms,
the homomorphisms $\psi$ and $\rho$
are transferred to 
\begin{align*}
\psi'&=\tau\inv\circ\psi\circ\sigma\:
(A\rtimes_\alpha G,\what\alpha,i_G^\alpha)
\to
(B\rtimes_\beta G,\what\beta,i_G^\beta)
\\
\rho'&=\omega\inv\circ\rho\circ\tau\:
(B\rtimes_\beta G,\what\beta,i_G^\beta)
\to
(C\rtimes_\gamma G,\what\gamma,i_G^\gamma)
\end{align*}
in $\wcod$,
and because the isomorphisms $\sigma$, $\tau$, and $\omega$ preserve all structure
we see that $\psi'$ and $\rho'$ will satisfy \eqref{morphism}.

Since $\wcpd$ is bijective between morphism sets in $\aco$ and sets of morphism sets in $\wcod$ determined by the condition \eqref{morphism},
there are unique morphisms
\begin{align*}
(\varphi,u)&\:(A,\alpha)\to (B,\beta)\\
(\sigma,v)&\:(B,\beta)\to (C,\gamma)
\end{align*}
in $\aco$ such that
\begin{align*}
\psi'&=\wcpd(\varphi,u)\\
\rho'&=\wcpd(\sigma,v).
\end{align*}
Since $\wcpd$ is functorial we have
\[
\rho'\circ\psi'=\wcpd\bigl((\sigma,v)\circ (\varphi,u)\bigr).
\]
Thus $\rho'\circ\psi'$ satisfies \eqref{morphism}.
Since the isomorphisms $\sigma$ and $\omega$ preserve all structure,
the morphism
\[
\rho\circ\psi=\sigma\inv\circ\rho'\circ\psi'\circ\omega
\]
also satisfies \eqref{morphism}.

Thus we have proved the first statement of the theorem,
establishing the existence of the subcategory $\wcoo$ of $\wcod$ consisting of the same objects but only those morphisms satisfying \eqref{morphism}.
In view of the bijections \eqref{bij},
it now follows that $\wcpd$ gives a full and faithful functor
\[
\wcpo\:\aco\to \wcoo,
\]
which is essentially surjective since $\wcpd$ is,
and therefore is a category equivalence.
\end{proof}

\begin{rem}
Regarding the bijection \eqref{bij},
of course the injectivity also follows from fidelity of the functor $\wcpd$ in \propref{dumb functor}.
On the other hand, the proof of surjectivity is the only place in the entire paper that the full strength of Pedersen's theorem is needed.
Thus, Pedersen's theorem is what guarantees that the functor $\wcpo\:\aco\to \wcoo$ is full.
\end{rem}

We define a forgetful functor $F\colon \wcoo\to \csn$
on objects just as for $F\colon \wcon\to \csn$,
and on morphisms by taking
\[
\psi\colon (C,\delta,V)\to (D,\epsilon,W)
\]
to the same map viewed as a morphism $C\to D$ in $\csn$.
We define the \emph{outer crossed-product functor} as the composition
\[
\cpo:=F\circ \wcpo\colon \aco\to \csn.
\]
This setting describes an inversion of $\cpo$ in the sense of \defnref{invert def},
which we call \emph{outer Landstad duality for actions}.

\begin{prop}\label{good ou}
The above outer Landstad duality for actions is a good inversion.
\end{prop}

\begin{proof}
We must check the unique isomorphism lifting property:
given an object $(C,\delta,V)$ of $\wcoo$ and an isomorphism $\theta\colon C\variso D$ in $\csn$,
since nondegenerate Landstad duality for actions is a good inversion by \propref{good nd},
we have extra structure $(\epsilon,W)$ for $D$ such that
$\theta$ gives an isomorphism
\[
\wilde\theta\colon (C,\delta,V)\variso (D,\epsilon,W)
\]
in $\wcon$, and hence an isomorphism in $\wcoo$,
covering $\theta$,
and since the forgetful functor is faithful
we see that $\wilde\theta$ is unique.
\end{proof}

\section{Inverting the crossed-product process --- coactions}\label{invert coactions}

The exposition we give below will parallel what we did in the preceding section for actions, especially for nondegenerate and enchilada dualities.
However, subsequent to \propref{copedersen}, we remarked that we do not know if the converse of Pedersen's theorem holds for coactions. Consequently, although we have complete versions of nondegenerate and enchilada dualities for coactions, we do not have an outer duality.

As before, in all three cases the objects of the categories $\CC,\DD,\wilde \DD$ will remain the same:
we start with a category $\CC$ of normal coactions,
the basic process will produce the crossed-product $C^*$-algebra,
and the objects in the  category $\wilde \DD$ will be equivariant actions.
When the development is exactly parallel to that in the preceding section, modulo a completely routine switching of ``action'' and ``normal coaction'', together with routine adjustments in the notation, we will merely mention the analogous results. However, computations involving coactions are frequently of a different character than those for actions, and in all appropriate cases we will include these computations.

\subsection{Nondegenerate Landstad duality for coactions}

The \emph{nondegenerate category $\con$ of coactions}
has
normal coactions $(A,\delta)$ of $G$ as objects,
and when we say
$\phi\colon (A,\delta)\to (B,\epsilon)$ is a morphism in the category we mean
$\phi\colon A\to B$ is a morphism in $\csn$
that is $\delta-\epsilon$ equivariant.
Isomorphisms in the category are equivariant $C^*$-isomorphisms.

The \emph{nondegenerate equivariant category $\wacn$ of actions}
has
equivariant actions (see \defnref{fix action}) $(C,\alpha,\mu)$ of $G$
as objects,
and when we say
$\psi\colon (C,\alpha,\mu)\to (D,\beta,\nu)$ is a morphism in the category we mean
$\psi\colon C\to D$ is a morphism in $\csn$
that is $\alpha-\beta$ equivariant
and satisfies
\[
\nu=\psi\circ \mu.
\]

The 
\emph{nondegenerate crossed-product functor $\cpn$} is given
on objects by
$(A,\delta)\mapsto A\rtimes_\delta G$,
and
on morphisms by
\[
\bigl(\phi\colon (A,\delta)\to (B,\epsilon)\bigr)
\mapsto
\bigl(\phi\rtimes G\colon A\rtimes_\delta G\to B\rtimes_\epsilon G\bigr).
\]

The \emph{nondegenerate equivariant crossed-product functor $\wcpn$} is given
on objects by
$(A,\delta)\mapsto (A\rtimes_\delta G,\what\delta,j_G)$,
and
on morphisms by
\begin{multline*}
\bigl(\phi\colon (A,\delta)\to (B,\epsilon)\bigr)
\mapsto
\\
\bigl(\phi\rtimes G\colon 
(A\rtimes_\delta G,\what\delta,j_G)\to (B\rtimes_\epsilon G,\what\epsilon,j_G\bigr).
\end{multline*}
The functor $\wcpn$ is an equivalence \cite[Theorem~4.2 and Corollary~4.3]{cldx} and
$\cpn$ is the composition of $\wcpn$ followed by the forgetful functor $F\colon\wacn\to\csn$
defined on objects by $(C,\alpha,\mu)\mapsto C$ and on morphisms by $f\mapsto f$.
Hence, $F$ is precisely the type of forgetful functor that fits into the framework of \secref{abstract},
and hence, this setup gives an inversion of the process $\cpn$.
We call this inversion \emph{nondegenerate Landstad duality for coactions}.

By \cite[Theorem~4.2 and Corollary~4.3]{cldx} a quasi-inverse of the nondegenerate equivariant crossed-product functor
is given by the \emph{nondegenerate fixed-point functor $\fin$}, given
on objects by
$(C,\alpha,\mu)\mapsto (C^{\alpha,\mu},\delta^\mu)$
(see \defnref{fix action} for the notation),
and
on morphisms 
as follows:
if $\psi\colon (C,\alpha,\mu)\to (D,\beta,\nu)$ is a morphism in $\wacn$,
then
\[
\fin(\psi)\colon (C^{\alpha,\mu},\delta^\mu)\to (D^{\beta,\nu},\delta^\nu)
\]
is the unique morphism in $\con$
such that
the diagram
\[
\xymatrix@C+30pt{
(C^{\alpha,\mu}\rtimes_{\delta^\mu} G,\what{\delta^\mu},j_{C^{\alpha,\mu}})
\ar[r]^-{\fin(\psi)\rtimes G} \ar[d]^\simeq
&(D^{\beta,\nu}\rtimes_{\delta^\nu} G,\what{\delta^\nu},j_{D^{\beta,\nu}})
\ar[d]_\simeq
\\
(C,\alpha,\mu) \ar[r]_-\psi
&(D,\beta,\nu)
}
\]
commutes in $\wacn$,
where the vertical arrows are the canonical isomorphisms.

Since we have chosen the object map of $\fin$ to take an equivariant action $(C,\alpha,\mu)$ to the $C^*$-subalgebra $C^{\alpha,\mu}$ of $M(C)$,
in our setting
the nondegenerate 
homomorphism
\[
\fin(\psi)\colon C^{\alpha,\mu}\to M(D^{\beta,\nu})
\]
is the restriction of (the canonical extension to $M(C)$ of) $\psi$.

Thus,
the additional data required to recover the coaction from the crossed product $A\rtimes_\delta G$ consists of the dual action $\what\delta$ and the canonical homomorphism $j_G$.

\begin{prop}\label{good cond}
The above nondegenerate Landstad duality is a good inversion.
\end{prop}

\begin{proof}
This is a routine adaptation from the action case.
\end{proof}

\subsection{Enchilada Landstad duality for coactions}

The \emph{enchilada category $\coe$ of coactions}
has
the same objects as $\con$,
but now when we say
$[X,\zeta]\colon (A,\delta)\to (B,\epsilon)$ is a morphism in the category we mean
$[X]\colon A\to B$ in $\cse$
and $\zeta$ is a $\delta-\epsilon$ compatible action of $G$ on $X$.
Isomorphisms in the category are equivariant Morita equivalences.

The \emph{enchilada equivariant category $\wace$ of actions}
has
the same objects as $\wacn$,
but now when we say
$[Y,\gamma]\colon (C,\alpha,\mu)\to (D,\beta,\nu)$ is a morphism in the category we mean
$[Y]\colon C\to D$ in $\cse$ and
$\gamma$ is a $\alpha-\beta$ compatible coaction of $G$ on $Y$.
The isomorphisms in the category are precisely the equivariant Morita equivalences of the actions.

The \emph{enchilada crossed-product functor $\cpe$} is
the same as $\cpn$
on objects,
but is given
on morphisms by
\[
\bigl([X,\zeta]\colon (A,\delta)\to (B,\epsilon)\bigr)
\mapsto
\bigl([X\rtimes_\zeta G]\colon A\rtimes_\delta G\to B\rtimes_\epsilon G\bigr).
\]

The \emph{enchilada equivariant crossed-product functor $\wcpe$} is 
the same as $\wcpn$
on objects,
but is given
on morphisms by
\begin{multline*}
\bigl([X,\zeta]\colon (A,\delta)\to (B,\epsilon)\bigr)
\mapsto
\\
\bigl([X\rtimes_\zeta G,\what\zeta]\colon 
(A\rtimes_\delta G,\what\delta,j_G)\to (B\rtimes_\epsilon G,\what\epsilon,j_G\bigr).
\end{multline*}

The following two results are routine modifications of the corresponding Propositions~\ref{fixed correspondence} and \ref{action up down}.

In the following proposition, the existence of the $C^\alpha-D^\beta$ correspondence $X$ is established (with greater generality) in \cite[Corollary~6.4]{BusEch2}. 
The construction, which is based upon a technique introduced in \cite{enchilada}, 
is essentially
the same 
one that
we used 
in \propref{fixed correspondence}.

\begin{prop}\label{fixed correspondence 2}
Let $(C,\alpha,\mu)$ and $(D,\beta,\nu)$ be equivariant actions,
and let $(Y,\gamma)$ be a $(C,\alpha)-(D,\beta)$ correspondence action.
Then there are a
$(C^\alpha,\delta^\mu)-(D^\beta,\delta^\nu)$
correspondence coaction $(X,\zeta)$
and an isomorphism
\[
\Theta\colon (X\rtimes_\zeta G,\what\zeta)\variso (Y,\gamma)
\]
of
$(C,\alpha)-(D,\beta)$ correspondence coactions,
characterized by
\[
\Theta\bigl(j_X(x)\cdot j_G(f)\bigr)=x\cdot \nu(f)
\midtext{for}x\in X,f\in C_0(G).
\]
\end{prop}

\begin{proof}
The argument is completely parallel to that of \propref{fixed correspondence},
and so we omit the details.
The only point we should mention is that
the decomposition of the associated inner coaction on the linking algebra (see \subsecref{linking})
$L(Y)=\smtx{K&Y\\{*}&D}$ takes the form
\[
\mtx{\ad \kappa&\ad (\kappa,\nu)\\{*}&\ad \nu},
\]
where 
$\kappa$ and $\nu$ are nondegenerate homomorphisms of $C_0(G)$ into $M(K)$ and $M(D)$, respectively, and
$\ad (\kappa,\nu)$ denotes the coaction of $G$ on $Y$
given by $y\mapsto \kappa\otimes\id(w_G)\cdot (y\otimes 1)\cdot \nu\otimes\id(w_G^*)$.
\end{proof}

\begin{notn}\label{fix corr def 2}
We denote the $C^\alpha-D^\beta$ correspondence $X$ constructed in the above proof
by $Y^{\gamma,\mu,\nu}$, or just $Y^\gamma$ if confusion is unlikely,
and we denote the coaction $\zeta$ by $\delta^{\mu,\nu}$.
\end{notn}

\begin{prop}\label{coaction up down}
Let $(A,\delta)$ and $(B,\epsilon)$ be normal coactions,
and let $(X,\zeta)$ be an $(A,\delta)-(B,\epsilon)$ correspondence coaction.
Then
$j_X\colon X\to M(X\rtimes_\zeta G)$ gives
an isomorphism
\[
(X,\zeta)
\iso 
\bigl((X\rtimes_\zeta G)^{\what\zeta},\delta^{j_G^\delta,j_G^\epsilon}\bigr)
\]
of $(A,\delta)-(B,\epsilon)$ correspondence coactions.
\end{prop}

\begin{proof}
The argument is completely parallel to that of \propref{action up down},
and so we omit the details.
\end{proof}

As for the case of actions, the above two results imply the following.

\begin{thm}\label{enchilada invert2}
The enchilada equivariant crossed-product functor $\wcpe$ is an equivalence,
and there is a quasi-inverse
$\fie\colon \wace\to \coe$ 
with
morphism map
\begin{multline*}
\bigl([Y,\gamma]\colon (C,\alpha,\mu)\to (D,\beta,\nu)\bigr)
\\\mapsto
\bigl([Y^\gamma,\delta^{\mu,\nu}]\colon (C^\alpha,\delta^\mu)\to (D^\beta,\delta^\nu)\bigr)
\end{multline*}
and the same object map as $\fin\colon \wacn\to \con$.
\end{thm}

\begin{proof}
The argument is completely parallel to that of \propref{enchilada invert},
and so we omit the details.
\end{proof}

Let $F\colon\wace\to\cse$ denote the forgetful functor defined on objects by $(C,\delta,V)\mapsto C$.
The factoring of $\cpe$ into a composition of $\wcpe$ followed by $F$ gives an inversion of $\cpe$,
which we call \emph{enchilada Landstad duality for actions}.

\begin{rem}
This inversion is not good,
since when $G$ is abelian the coactions become actions of the dual group,
and we have observed earlier that enchilada Landstad duality for actions is not a good inversion.
Nevertheless, just as for actions, the forgetful functor is faithful and essentially surjective
(because we can again use crossed-product duality).
\end{rem}

\begin{rem}
It might be of interest to note that our use of category-theory technique in the above proof obviated the need to directly establish that the morphism map $[Y,\gamma]\mapsto [Y^\gamma,\delta^{\mu,\nu}]$ is functorial; this would have required that we prove an isomorphism of the form
\[
(Y\otimes_D Z)^{\gamma\otimes \rho}\simeq Y^\gamma\otimes_{D^\beta} Z^\rho,
\]
whereas in fact this follows from the properties of category equivalences.
In contrast, the functoriality in \cite[Corollary~6.4]{BusEch2} depends upon \cite[Proposition~6.1]{BusEch2}, which proves such a tensor-product isomorphism; this was necessary in \cite{BusEch2} because their fixed-point correspondence functor was not presented as a quasi-inverse to a known functor.
\end{rem}

\subsection{Outer Landstad duality for coactions}

The \emph{outer category $\coo$ of coactions}
has the same objects as $\con$,
but now when we say $(\phi,U)\colon (A,\delta)\to (B,\epsilon)$ is a morphism in the category
we mean $U$ is an $\epsilon$-cocycle and
$\phi\colon A\to B$ is a morphism in $\csn$
that is $\delta-\ad U\circ \epsilon$ equivariant.

\begin{lem}
The category $\coo$ introduced above is well-defined.
\end{lem}

\begin{proof}
The outline of the proof is completely parallel to that of \lemref{outer well defined};
we only include 
those calculations
that are peculiar to coactions.
The crucial thing is to check that 
we can compose morphisms:
given morphisms
$(\phi,U)\colon (A,\alpha)\to (B,\beta)$ and $(\psi,V)\colon (B,\beta)\to (C,\zeta)$ 
in $\coo$,
we must show
\begin{enumerate}
\item
$(\psi\otimes\id)(U)V$ is a $\zeta$-cocycle, and

\item
$\psi\circ\phi$ is $\delta-\ad \bigl[(\psi\otimes\id)(U)V\bigr]\circ\zeta$ equivariant.
\end{enumerate}

For (i), 
as
we show below, $(\psi\otimes\id)(U)$ is an $\ad V\circ\zeta$ cocycle, and hence it follows from \lemref{cocycle multiply coaction} that
$(\psi\otimes\id)(U)V$ is a $\zeta$-cocycle.

For (ii), we appeal to \cite[Remark~1.14]{fischer},
concerning naturality of cocycles:
$\psi$ is $\epsilon-\ad V\circ\zeta$ equivariant
and $U$ is an $\epsilon$-cocycle, so
$(\psi\otimes\id)(U)$ is an $\ad V\circ\zeta$ cocycle
and $\psi$ is
\[
\ad U\circ\epsilon-\ad (\psi\otimes\id)(U)\circ \ad V\circ\zeta
\]
equivariant.
Note that
\[
\ad (\psi\otimes\id)(U)\circ \ad V\circ\zeta=\ad [(\psi\otimes\id)(U)V]\circ\zeta.
\]
Since $\phi$ is $\delta-\ad U\circ\epsilon$ equivariant,
the composition $\psi\circ\phi$ is
\[
\delta-\ad \bigl[(\psi\otimes\id)(U)V\bigr]\circ\zeta
\]
equivariant.
This proves the claim, and so composition of morphisms is well-defined.

It is obvious that there are identity morphisms,
and a routine computation shows that composition is associative.
\end{proof}

Isomorphisms in the category are just outer conjugacies of normal coactions.

Parallel to the fixed-point equivariant category of coactions,
we would now like to define
the \emph{fixed-point equivariant category $\waco$ of actions}
in which the objects are the same as in $\wacn$,
and in which a morphism
$\psi\colon (C,\alpha,\mu)\to (D,\beta,\nu)$ 
is a morphism
$\psi\colon C\to D$ 
in $\csn$
that is $\alpha-\beta$ equivariant
and satisfies
\begin{equation}\label{first kind}
D^{\beta,\nu}=D^{\beta,\psi\circ \mu}.
\end{equation}
However, just as before we do not see how to prove directly that composition of morphisms will be well-defined.
Moreover, in this case we cannot use the indirect approach that we did for $\wcoo$
because we have no fully working version of Pedersen's theorem for outer conjugacy of coactions.

Consequently,
in this case we need to modify the definition of morphisms in the category.
In fact, we will replace \eqref{first kind} by a condition that is
formally weaker (see \remref{weaker}),
but for all we know the
two definitions are
equivalent ---
fortunately, our result will not depend upon the answer to this question.

So, we start over:
we define the \emph{fixed-point equivariant category $\waco$ of actions}
to have the same objects as $\wacn$,
namely equivariant actions,
and
when we say $\psi\:(C,\alpha,\mu)\to (D,\beta,\nu)$ is a \emph{morphism} in $\waco$,
we mean that $\phi\:C\to D$ is a morphism in $\csn$
that is $\alpha-\beta$ equivariant and for which
the canonical extension
\[
\bar\phi\:M(C)\to M(D)
\]
restricts to a nondegenerate homomorphism
\[
\bar\phi|\:C^{\alpha,\mu}\to M(D^{\beta,\nu}).
\]

\begin{lem}\label{compose}
With the above definition of morphism, the category $\waco$ is well-defined.
\end{lem}

\begin{proof}
We must check that composition of morphisms is defined.
Once we have done this it will be obvious that composition is associative and that we have identity morphisms.
Suppose that $\rho\:(D,\beta,\nu)\to (E,\gamma,\tau)$ is another morphism,
so that $\bar\rho$ restricts to a nondegenerate homomorphism
\[
\bar\rho|\:D^{\beta,\nu}\to M(E^{\gamma,\tau}).
\]
The composition of $\psi$ and $\rho$ in $\csn$
is the $\alpha-\gamma$ equivariant nondegenerate homomorphism
\[
\bar\rho\circ\psi\:C\to M(E).
\]
On the other hand,
the composition of the nondegenerate homomorphisms
$\bar\psi|$ and $\bar\rho|$
is the nondegenerate homomorphism
\[
\bar{\bar\rho|}\circ\bar\psi|\:C^{\alpha,\mu}\to M(E^{\gamma,\tau}).
\]
It is clear from the definitions that this composition is the restriction of
$\bar{\bar\rho\circ\psi}=\bar\rho\circ\bar\psi$
to $C^{\alpha,\mu}$.
\end{proof}

\begin{rem}
If $\psi\:(C,\alpha,\mu)\to (D,\beta,\nu)$ is a morphism in the nondegenerate category $\wacn$ of equivariant actions,
then we can apply the functor
$\fin$ to get the nondegenerate homomorphism
\[
\fin(\psi)=\bar\psi|\:C^{\alpha,\mu}\to M(D^{\beta,\nu}).
\]
Thus $\wacn$ is a subcategory of $\waco$ with the same objects.
\end{rem}

\begin{rem}\label{weaker}
Given equivariant actions $(C,\alpha,\mu)$ and $(D,\beta,\nu)$,
if $\psi\:C\to D$ is an $\alpha-\beta$ equivariant morphism in $\csn$
that satisfies \eqref{first kind},
then $\psi$ is a morphism in $\waco$,
because we have a morphism
\[
\psi\:(C,\alpha,\nu)\to (D,\beta,\psi\circ\mu)
\]
in $\wacn$, and hence in $\waco$,
i.e., $\bar\psi$ restricts to a nondegenerate homomorphism
\[
\bar\psi|\:C^{\alpha,\mu}\to M(D^{\beta,\psi\circ\mu})=M(D^{\beta,\nu}).
\]
We do not know whether the converse holds ---
for all we know, the morphisms in $\waco$ might be precisely those equivariant nondegenerate homomorphisms satisfying \eqref{first kind},
completely parallel with $\wcoo$.
We should mention that it is not hard to show that the converse does hold for isomorphisms.
\end{rem}

\begin{thm}\label{waco}
With the above notation,
the assignments
\begin{align*}
(A,\delta)&\mapsto (A\rtimes_\delta G,\what\delta,j_G^\delta)
\\
(\phi,U)&\mapsto \Phi_U\circ (\phi\rtimes G)
\end{align*}
give a functor
$\wcpo\colon \coo\to \waco$
that is essentially surjective and faithful.
\end{thm}

\begin{proof}
The first thing to check is that 
if $(\phi,U)\:(A,\delta)\to (B,\epsilon)$ is a morphism in $\coo$
then
\[
\wcpo(\phi,U)=\Phi_U\circ(\phi\rtimes G)\:(A\rtimes_\delta G,\what\delta,j_G^\delta)
\to (B\rtimes_\epsilon G,\what\epsilon,j_G^\epsilon)
\]
is a morphism in $\waco$.
Let $\zeta=\ad U\circ\epsilon$.
Then we have a morphism
\[
\varphi\rtimes G\:(A\rtimes_\delta G,\what\delta,j_G^\delta)
\to (B\rtimes_\zeta G,\what\zeta,j_G^\zeta)
\]
in $\wacn$, and hence in $\waco$.
Next, we have a $\what\zeta-\what\epsilon$ equivariant morphism
\[
\Phi_U\:B\rtimes_\zeta G\to B\rtimes_\epsilon G
\]
in $\csn$,
and we have a nondegenerate homomorphism
\[
\Phi_U\circ j_B^\zeta=j_B^\epsilon\:B\to M(B\rtimes_\epsilon G),
\]
and so $\bar{\Phi_U}$ restricts to a nondegenerate homomorphism
from
\[
j_B^\zeta(B)=(B\rtimes_\zeta G)^{\what\zeta,j_G^\zeta}
\]
to
\[
M(j_B^\epsilon(B))=M\bigl((B\rtimes_\epsilon G)^{\what\epsilon,j_G^\epsilon}\bigr).
\]
Thus we have a morphism
\[
\Phi_U\:(B\rtimes_\zeta G,\what\zeta,j_G^\zeta)\to (B\rtimes_\epsilon G,\what\epsilon,j_G^\epsilon)
\]
in $\waco$.
Therefore the composition
\[
\wcpo(\phi,U)=\Phi_U\circ (\varphi\rtimes G)
\]
is a morphism in $\waco$.

Now let $(\psi,V)\:(B,\epsilon)\to (C,\zeta)$ be another morphism in $\coo$.
Then
\begin{align*}
&\wcpo(\psi,V)\circ \wcpo(\phi,U)\circ j_A
=\wcpo\bigl((\psi,V)\circ (\phi,U)\bigr)\circ j_A
\end{align*}
by computations parallel to \propref{dumb functor}.
On the other hand,
\begin{align*}
&\bigl(\wcpo(\psi,V)\circ \wcpo(\phi,U)\circ j_G^\delta\otimes\id\bigr)(w_G)
\\&\quad=\bigl(\Phi_V\circ(\psi\rtimes G)\circ \Phi_U\circ(\phi\rtimes G)\circ j_G^\delta\otimes\id\bigr)(w_G)
\\&\quad=\bigl(\Phi_V\circ(\psi\rtimes G)\otimes\id\bigr)
\bigl((\Phi_U\circ j_G^{\ad U\circ\epsilon}\otimes\id)(w_G)\bigr)
\\&\quad=\bigl(\Phi_V\circ(\psi\rtimes G)\otimes\id\bigr)
\bigl((j_B^\epsilon\otimes\id)(U)(j_G^\epsilon\otimes\id)(w_G)\bigr)
\\&\quad=\bigl(\Phi_V\circ(\psi\rtimes G)\circ j_B^\epsilon\otimes\id\bigr)(U)
\bigl(\Phi_V\circ(\psi\rtimes G)\circ j_G^\epsilon\otimes\id\bigr)(w_G)
\\&\quad=\bigl(\Phi_V\circ j_C^{\ad V\circ\zeta}\circ\psi\otimes\id\bigr)(U)
\bigl(\Phi_V\circ j_G^{\ad V\circ\zeta}\otimes\id\bigr)(w_G)
\\&\quad=(j_C^\zeta\circ\psi\otimes\id)(U)(j_C^\zeta\otimes\id)(V)(j_G^\zeta\otimes\id)(w_G)
\\&\quad=(j_C^\zeta\otimes\id)\bigl((\psi\otimes\id)(U)V\bigr)(j_G^\zeta\otimes\id)(w_G)
\\&\quad=\bigl(\Phi_{(\psi\otimes\id)(U)V}\circ(\psi\circ\phi\rtimes G)\circ j_G^\delta\otimes\id\bigr)(w_G)
\\&\quad=\Bigl(\wcpo\bigl(\psi\circ\phi,(\psi\otimes\id)(U)V\bigr)\circ j_G^\delta\otimes\id\Bigr)(w_G)
\\&\quad=\Bigl(\wcpo\bigl((\psi,V)\circ (\phi,U)\bigr)\circ j_G^\delta\otimes\id\Bigr)(w_G),
\end{align*}
which implies
\[
\wcpo(\psi,V)\circ \wcpo(\phi,U)\circ j_G^\delta
=\wcpo\bigl((\psi,V)\circ (\phi,U)\bigr)\circ j_G^\delta.
\]

For fidelity,
given morphisms
\[
(\phi,U),(\rho,V)\colon (A,\delta)\to (B,\epsilon)
\]
in $\coo$ such that
\[
\wcpo(\phi,U)=\wcpo(\rho,V)\colon 
(A\rtimes_\delta G,\what\delta,j_G^\delta)\to
(B\rtimes_\epsilon G,\what\epsilon,j_G^\epsilon)
\]
in $\waco$,
we have
\begin{align*}
\wcpo(\phi,U)\circ j_A^\delta
&=\Phi_U\circ(\phi\rtimes G)\circ j_A^\delta
\\&=\Phi_U\circ j_B^{\ad U\circ\epsilon}\circ\phi
\\&=j_B^\epsilon\circ \phi,
\end{align*}
and similarly
\[
\wcpo(\rho,V)\circ j_A^\delta=j_B^\epsilon\circ \rho,
\]
so $\phi=\rho$ since $j_B^\epsilon$ is injective by normality of $\epsilon$.

On the other hand,
\begin{align*}
\bigl(\wcpo(\phi,U)\circ j_G^\delta\otimes\id\bigr)(w_G)
&=\bigl(\Phi_U\circ(\phi\rtimes G)\circ j_G^\delta\otimes\id\bigr)(w_G)
\\&=(j_B^\epsilon\otimes\id)(U)(j_G^\epsilon\otimes\id)(w_G),
\end{align*}
and similarly
\[
\bigl(\wcpo(\rho,V)\circ j_G^\delta\otimes\id\bigr)(w_G)
=(j_B^\epsilon\otimes\id)(V)(j_G^\epsilon\otimes\id)(w_G),
\]
so $(j_B^\epsilon\otimes\id)(U)=(j_B^\epsilon\otimes\id)(V)$ since $(j_G^\epsilon\otimes\id)(w_G)$ is unitary.
Since $j_B^\epsilon$ is injective by normality of the coaction $\epsilon$, so is
\[
j_B^\epsilon\otimes\id\colon B\otimes C^*(G)\to M\bigl((B\rtimes_\epsilon G)\otimes C^*(G)\bigr).
\]
Thus $U=V$, and hence $(\phi,U)=(\rho,V)$.
\end{proof}

\begin{rem}
It can be shown by a computation similar to the first part of the proof of \thmref{waco} that if $\phi$ is a morphism in $\aco$ then the morphism $\wcpo(\phi)$ in $\wcoo$ actually satisfies the formally stronger condition \eqref{first kind}. Thus, if
we had a converse of Pedersen's theorem for coactions,
then we would 
be able to (re)define a category
$\waco$
subject to \eqref{first kind},
and get a category equivalence $\coo\sim \waco$ 
dual to the equivalence $\aco\sim \wcoo$ of
\subsecref{outer duality actions}.
Moreover, this 
would give rise to an outer Landstad duality for coactions,
which
would be a good inversion,
since we could show that the forgetful functor
has the required properties using a routine adaption of the argument for outer duality for actions.
\end{rem}

\begin{rem}
In \secref{invert actions} we worked with full crossed products by actions, and consequently the dual coactions were maximal. As we remarked at the end of that section, it is possible to give an alternative development, involving reduced crossed products, in which case the dual coactions would be normal.
In \secref{invert coactions} we chose to work with crossed products by normal coactions, rather than maximal ones,
because this allowed for a development that was quite parallel to the one in \secref{invert actions}.
It is possible to prove analogues of the results of \secref{invert coactions} for maximal coactions,
but this requires substantial modification of the techniques, principally because the generalized fixed-point algebras will then be in a different place. We felt that to present all of this here would distract from the main point, namely the description of a general procedure for ``inverting the process''.
We will give an alternative development in terms of maximal coactions in a forthcoming paper.
\end{rem}


\providecommand{\bysame}{\leavevmode\hbox to3em{\hrulefill}\thinspace}
\providecommand{\MR}{\relax\ifhmode\unskip\space\fi MR }
\providecommand{\MRhref}[2]{%
  \href{http://www.ams.org/mathscinet-getitem?mr=#1}{#2}
}
\providecommand{\href}[2]{#2}

\end{document}